\documentclass[a4paper,10pt]{article}
\usepackage[utf8x]{inputenc}
\pdfoutput=1

\usepackage{nccfoots}
\usepackage{times}
\usepackage{geometry}
\usepackage[T1]{fontenc}
\usepackage{color}
\usepackage{amsmath}
\usepackage{amssymb}
\usepackage{array}
\usepackage{amsthm}
\usepackage{graphicx}
\usepackage{hyperref}
\usepackage{setspace}
\usepackage{stmaryrd}
\usepackage{longtable}
\usepackage{tabularx}
\usepackage{multirow}
\usepackage{caption}
\usepackage{esint}

\theoremstyle{plain}
\newtheorem{thm}{Theorem}[section]

\newtheorem{cor}[thm]{Corollary}
\newtheorem{lem}[thm]{Lemma}

\theoremstyle{definition}
\newtheorem{defn}[thm]{Definition}

\theoremstyle{remark}
\newtheorem{rem}[thm]{Remark}

\theoremstyle{plain}

\newcommand{\R}{\mathbb{R}}

\newcommand{\N}{\mathbb{N}}

\newcommand{\dimn}{\mathrm{dim}}

\newcommand{\supp}{\mathrm{supp}}

\newcommand{\scal}{\mathrm{scal}}
\newcommand{\ric}{\mathrm{Ric}}
\newcommand{\trace}{\mathrm{tr}}
\newcommand{\kernel}{\mathrm{ker}}

\newcommand{\volume}{\mathrm{vol}}
\newcommand{\mult}{\mathrm{mult}}
\newcommand{\dv}{\text{ }dV}

\newcommand{\spectrum}{\mathrm{spec}}
\newcommand{\gradient}{\mathrm{grad}}

\newcommand{\ol}{\overline}
\newcommand{\wt}{\widetilde}
\newcommand{\wh}{\widehat}

\renewcommand{\title}[1]{{\bfseries #1}\par}
\renewcommand{\author}[1]{\medskip{#1}\par\smallskip}
\newcommand{\affiliation}[1]{{\itshape #1}\par}
\newcommand{\email}[1]{E-mail:~\texttt{#1}\par}

\numberwithin{equation}{section}

\begin{document}
\begin{center}
\title{\LARGE Spectra, rigidity and stability of sine-cones}
\vspace{3mm}
\author{\Large Klaus Kröncke}
\vspace{3mm}
\affiliation{Universität Hamburg, Fachbereich Mathematik\\Bundesstraße 55\\20146 Hamburg, Germany}

\email{klaus.kroencke@uni-hamburg.de} 
\end{center}
\vspace{2mm}
\begin{abstract}
We compute the spectra of the Laplace-Beltrami operator, the connection Laplacian on $1$-forms and the Einstein operator on symmetric $2$-tensors on the sine-cone over a positive Einstein manifold $(M,g)$. We conclude under which conditions on $(M,g)$, the sine-cone is dynamically stable under the singular Ricci-de Turck flow and rigid as a singular Einstein manifold. 
\end{abstract}
\section{Introduction and main results}
A Riemannian  manifold $(M,g)$ is called Einstein, if $\ric_g=\lambda\cdot g$ for some $\lambda\in\R$. Einstein manifolds play an important role in Riemannian geometry as well as in theoretical physics, especially if they are of special holonomy (e.g.\ Calabi-Yau).
The following two questions concerning \emph{rigidity} and \emph{stability} are central in the study of an Einstein manifold $(M,g)$:
\begin{itemize}
\item Can $g$ be deformed to another Einstein manifold which is not isometric or homothetic to $g$?
\item Is $(M,g)$ dynamically stable under the Ricci flow as a stationary point on the space of metrics modulo homotheties?
\end{itemize}
Both questions were treated extensively in the past on closed manifolds,
see \cite{Kro16} for an overview by the author with an exhaustive list of references.
Some more recent work has been done on noncompact manifolds e.g.
in negatively curved situations by Bamler \cite{Bam14,Bam15} and
 in the ALE setting \cite{DK20,KP20} by the author in joint work with Deruelle and Lindblad Petersen, respectively. Both questions become much more delicate in the noncompact setting, as they also depend in a subtle way on the fallof behaviour at infinity of perturbations of the metric $g$.
In the singular setting, these properties were studied only very recently in a paper by Boris Vertman and the author \cite{KV19}, where compact Ricci-flat conifolds were considered. 

As it appears in the linearization of the Einstein condition, an elliptic operator called the Einstein operator plays an important role in studying these questions. In particular, the properties of its spectrum may give an immediate answer to both questions, at least in some geometric situations. However, this operator is hard to investigate due to the complicated structure of the bundle on which it is acting. In particular, its spectrum is only known explicitly for very few examples and even achieving good eigenvalue bounds is a very hard problem. 

In this paper, we compute the spectrum of the Einstein operator on the sine-cone $(\wt{M},\wt{g})$ over a closed manifold $(M,g)$ in terms of the respective data on $(M,g)$. This class of singular manifolds is interesting in this context for various reasons. Most importantly, sine-cones appear as limits of smooth Einstein spaces.
The sine-cone over $S^n\times S^m$ is a Gromov-Hausdorff limit of smooth Einstein metrics on $S^{n+1}\times S^m$ \cite{Boh98} if $n+m\leq 9$. These are exactly the dimensions in which its Einstein operator is unbounded below. 
This suggests a deep general relation between the existence of smooth resolutions of Einsteinian sine-cones and spectral bounds on the Einstein operator. 

Sine-cones are also of great importance in theoretical physics \cite{BILPS14,GLNP11}. Moreover, we get many more examples of Einstein spaces on which spectra and eigensections are explicitly known. In fact, it seems to the best of our knowledge that this is the first paper, in which spectra of singular manifolds are computed explicitly.
\subsection{Spectra of geometric operators}

The sine-cone of a Riemannian manifold $(M,g)$ is given by
\begin{align}\label{eq:sinecone}
(\widetilde{M},\wt{g})=((0,\pi)\times M,d\theta^2+\sin(\theta)^2g).
\end{align}
It is one of the simplest examples of a Riemannian manifold with isolated conical singularities.
\noindent
On a sine-cone, we want to determine the spectrum of the following elliptic operators:
\begin{itemize}
\item The Laplace-Beltrami operator $\Delta:=\Delta_0:=-\trace{\nabla^2}:C^{\infty}(M)\to C^{\infty}(M)$, defined with the sign convention such that $\Delta f=-\trace{\nabla^2}f=-g^{ij}\nabla^2_{ij}f$.
\item The connection Laplacian $\Delta_1=-\trace{\nabla^2}:C^{\infty}(T^*M)\to C^{\infty}(T^*M)$ on $1$-forms.
\item The Einstein operator $\Delta_E=-\trace{\nabla^2}-2\mathring{R}:C^{\infty}(S^2M)\to C^{\infty}(S^2M)$ on symmetric $2$-tensors. Here, $\mathring{R}h_{ij}=g^{km}g^{ln}R_{iklj}h_{mn}$, with $R_{iklj}=g(R_{\partial_i,\partial_k}\partial_l,\partial_j)=g(\nabla^2_{\partial_i,\partial_j}\partial_k-\nabla^2_{\partial_j,\partial_i}\partial_k,\partial_l)$.
\end{itemize}
Note that the Einstein operator is closely related to the Lichnerowicz Laplacian, which is given by
\begin{align*}
\Delta_Lh:=-\trace{\nabla^2}h+\ric\circ h+h\circ \ric-2\mathring{R}.
\end{align*}
In particular, they coincide on Ricci-flat manifolds and differ only by an explicit constant on Einstein manifolds.\\
Spectral properties of self-adjoint Laplace-type operators on compact manifods with isolated conical singularities are by now well understood. In particular, if such an operator is bounded below, it admits a discrete spectrum and its normalized eigenvectors form an orthonormal basis of $L^2$. This follows from arguments as in the smooth case, using mapping properties of the resolvent and compact embedding properties of Sobolev spaces. For more background on the spectral theory of singular spaces, see e.g.\ \cite{Che83,BS87} and references therein.

From now on, we denote the spectrum of a differential operator $P$  by $\spectrum(P)$ and its positive part  by $\spectrum_+(P)=\spectrum(P)\cap(0,\infty)$. If $\lambda$ is an eigenvalue of the operator $P$, we denote the corresponding eigenspace by $E(P,\lambda)$.
\begin{thm}\label{mainthm:spectrum}Let $(M^n,g)$, $n\geq3$ be a positive Einstein manifold, scaled such that $\ric=(n-1)\cdot g$.
Given $\spectrum(\Delta_0)$, $\spectrum(\Delta_1)$ and $\spectrum(\Delta_E)$ on $(M,g)$, we explicitly compute \begin{itemize}
\item[(i)] $\spectrum(\wt{\Delta}_0)$ as a function of $\spectrum({\Delta}_0)$,
\item[(ii)] $\spectrum(\wt{\Delta}_1)$ as a function of $\spectrum({\Delta}_1)$ and $\spectrum({\Delta}_0)$,
\item[(iii)] $\spectrum(\wt{\Delta}_E)$ as a function of $\spectrum({\Delta}_E)$, $\spectrum({\Delta}_1)$ and $\spectrum({\Delta}_0)$.
\end{itemize}
\end{thm}
For more details, see Theorem \ref{sinconefunctions}, Theorem \ref{sincone1forms} and Theorem \ref{Einsteinspectrumsincone} below, where the eigenvalues and their multiplicities are explicitly given.
Our main interest lies in the spectrum of the Einstein operator. However, it contains the spectrum of the other two operators introduced above, which is why we have to compute the spectrum of all three of them.
The method we develop to compute these spectra is based on the following idea: We define two auxilliary manifolds, which are given by
\begin{align}\label{olMandwhM}
(\overline{M},\ol{g})=(\R_+\times M,dr^2+r^2g),\qquad (\wh{M},\wh{g})=(\R\times \overline{M},dz^2+\ol{g}).
\end{align}
It is now important to observe that by setting $r=\sin(\theta)\cdot s$ and $z=\cos(\theta)\cdot s$, we get an isometry
\begin{align}\label{isom}
 (\wh{M},\wh{g})\cong (\R_+\times \widetilde{M},ds^2+s^2\wt{g}).
\end{align}
Thus, the two auxiliary manifolds are just the cones over $(M,g)$ and its sine-cone.
Note that if $(M^n,g)$ is Einstein with $\ric_g=(n-1)g$, then $\ric_{\wt{g}}=n\cdot \wt{g}$, while $(\overline{M},\ol{g})$ and $ (\wh{M},\wh{g})$ are both Ricci flat. 

In the  proof of the main theorem, we establish a correspondence between eigensections on $(M,g)$ and homogeneous harmonic sections on $(\overline{M},\ol{g})$ as well a between eigensections on $(\widetilde{M},\wt{g})$ and homogeneous harmonic sections on $ (\wh{M},\wh{g})$. The identification \eqref{isom} allows us to also construct all homogeneous harmonic sections on $(\wh{M},\wh{g})$ from homogeneous harmonic sections on $(\overline{M},\ol{g})$, respectively.

In the case of functions, this correspondence will be carried out quite straightforwardly to determine the Laplace spectrum of $(\wt{M},\wt{g})$ in terms of the Laplace spectrum of $(M,g)$. For $1$-forms and in particular for symmetric $2$-tensors, the proof is much more involved. In those cases, the bundle splits into several parts where for each of these parts the above correspondence between eigensections  and homogeneous harmonic sections looks  different. Especially the correspondence of between 
homogeneous harmonic sections on $(\overline{M},\ol{g})$and $(\wh{M},\wh{g})$ is quite complicated and needs a large amount of tedious calculations.

This approach could also be used to compute the spectrum of other important geometric operators on singular spaces, e.g.\ the Hodge Laplacian on differential forms, the classical Dirac operator on spinors and twised versions of it. This potentially leads to important applications in index theory and analytic torsion of singular spaces.

\subsection{Stability of sine-cones}
There are various notions of stability for an Einstein manifold. For the following definition, we introduce the spaces
\begin{align*}
g^{\perp}&=\left\{h\in C^{\infty}(S^2M)\mid \int_M \trace h\dv=0\right\},\qquad
TT&=\left\{h\in C^{\infty}(S^2M)\mid \trace h=0,\delta h=0\right\}\subset g^{\perp}.
\end{align*}
Here, $\trace$ denotes the trace and $\delta$ denotes the divergence. Elements of $g^{\perp}$ are volume-preserving perturbations of $g$. Elements in $TT$ are called transverse traceless tensors or TT-tensors for short. Both spaces are invariant under $\Delta_E$, if $g$ is an Einstein metric.
Recall that $\ric_g=(n-1)g$ implies $\ric_{\ol{g}}=0$, where  
$\ol{g}$ is the cone metric defined in \eqref{olMandwhM}. 
Via parallel transport along radial lines, we can identify the bundle restrictions $S^2\ol{M}|_{\left\{r\right\}\times M})$ with each other. Therefore, we get a
 natural identification
\begin{align*}
C^{\infty}(S^2\ol{M})\cong C^{\infty}(\R_{+},C^{\infty}(S^2\ol{M}|_{\left\{1\right\}\times M})),
\end{align*}
with respect to which the Einstein operator of $\ol{g}$ is of the form
\begin{align*}
\ol{\Delta}_E=-\partial^2_{rr}-\frac{n-1}{r}\partial_r+\frac{1}{r^2}\Box_E,\qquad
 \Box_E: C^{\infty}(S^2\ol{M}|_{\left\{1\right\}\times M})
\to C^{\infty}(S^2\ol{M}|_{\left\{1\right\}\times M}).
\end{align*}
The self-adjoint elliptic operator $\Box_E$ is called the tangential operator of the Einstein operator. There is also a bundle isomorphism
\begin{align*}
S^2\ol{M}|_{\left\{1\right\}\times M}&\cong (M\times\R)\oplus T^*M\oplus S^2M,\\
h&\mapsto ((\pi(h),h(dr,dr)),h(dr,.)|_{M},h(.,.)|_{M\times M}),
\end{align*}
where $M\times\R$ is the trivial line bundle and $\pi$ denotes the projection onto the base point. 
Therefore, $\Box_E$ can be entirely regarded as an operator on $M$, without having constructed $\ol{M}$.
\begin{defn}\label{def:stability}We call an Einstein manifold $(M^n,g)$ with 
$\ric_g=(n-1)g$
\begin{itemize}
\item[(i)] Einstein-Hilbert stable, if $\Delta_E|_{TT}\geq 0$ and strictly Einstein-Hilbert stable, if $\Delta_E|_{TT}> 0$.
\item[(ii)] linearly stable, if $\Delta_E|_{g^{\perp}\cap\delta^{-1}(0)}\geq 0$ and strictly linearly stable, if $\Delta_E|_{g^{\perp}\cap\delta^{-1}(0)}> 0$.
\item[(iii)] tangentially stable, if $\Box_E\geq 0$ and  strictly tangentially stable, if $\Box_E> 0$.
\item[(iv)] physically stable, if $\Delta_E|_{TT}\geq-\frac{1}{4}(n-1)^2$.
\end{itemize}
\end{defn}

\begin{rem}
Let us give a brief motivation of these definitions: 
\begin{itemize}
\item[(i)]The notion of Einstein-Hilbert stability is coming from the variational theory of the Einstein-Hilbert functional. In fact, the second variation of the Einstein-Hilbert functional  $S$ restricted to $TT$-tensors is given by
\begin{align*}
\frac{d^2}{dt^2}|_{t=0}S(g+th)=-\frac{1}{2}\int_M \langle \Delta_Eh,h\rangle \dv,
\end{align*}
see \cite[Theorem 4.60]{Bes08}.
From now on, we will abbreviate this notion by (strict) EH-stability for convenience. 
\item[(ii)] The notion of  linear stability comes from the variational theory of Perelman's shrinker entropy $\nu_{-}$ and is more convenient in the context of Ricci flow. The second variation of $\nu_{-}$ restricted to $g^{\perp}\cap\ker(\delta)$ is given by
\begin{align*}
\frac{d^2}{dt^2}|_{t=0}\nu(g+th)=-\frac{1}{4(n-1)\volume(M)}\int_M \langle \Delta_Eh,h\rangle \dv,
\end{align*}
see \cite[Proposition 8.2]{Kro13}.
The restriction to $g^{\perp}\cap\ker(\delta)$ means that we only
consider essential deformations orthogonal to rescalings and pullbacks along diffeomorphisms.
\item[(iii)] In \cite{Ver16}, Vertman introduced a Ricci-de Turck flow on compact manifolds with isolated conical singularities that preserves the singularities along the flow. In order to get a well-posed initial value problem, Vertman assumed the singularities to be modelled over Ricci-flat cones, which are either strictly tangentially stable or flat orbifolds.
\item[(iv)] The spectral bound defining physical stability appears as a stability condition in the study of supersymmetric backgrounds, e.g.\ AdS-product spacetimes and higher-dimensional black holes \cite[Section II B and C]{GHP03}. 
Moreover, the author showed in \cite[Theorem 1.1]{Kro15c} that a Ricci flat cone $(\ol{M},\ol{g})$ over $(M,g)$ is linearly stable if and only if $(M,g)$ is physically stable. In all cases, the different stability conditions are related via the Hardy inequality and it directly appears in the proof of \cite[Theorem 1.1]{Kro15c}.
\end{itemize}
\end{rem}
\begin{rem}\label{rem:stability}
We have the following equivalences, see e.g.\ \cite[Corollary 1.3]{CH15} and \cite[Theorem 1.3]{KV19}:
\begin{itemize}
\item[(i)](strict) Linear stability $\Leftrightarrow$ [(strict) EH-stability $\wedge$ $\spectrum_+(\Delta)\geq  2(n-1)$ (resp.\ $> 2(n-1)$)].
\item[(ii)](strict) tangential stability $\Leftrightarrow$ [(strict) EH-stability $\wedge$ $\spectrum_+(\Delta)\geq 2(n+1)$ (resp.\ $> 2(n+1)$)].
\end{itemize}
\end{rem}
Using these definitions, we conclude the following from Theorem \ref{mainthm:spectrum}:
\begin{thm}\label{mainthm:linstability}
Let $(M^n,g)$, $n\geq3$ be a closed Einstein manifold with Einstein constant $n-1$ and $(\wt{M},\wt{g})$ be its sine-cone. Then the following equivalences do hold:
\begin{itemize}
\item[(i)] $(\wt{M},\wt{g})$ (strictly) EH-stable $\Leftrightarrow$  $(M,g)$ is (strictly) EH-stable.
\item[(ii)] $(\wt{M},\wt{g})$ (strictly) linearly stable $\Leftrightarrow$
$(M,g)$ (strictly) linearly stable and the Laplace spectrum satisfies
 $\spectrum_+(\Delta)\geq 2n-\frac{n}{2}(\sqrt{1+\frac{8}{n}}-1)$  (resp.\ $\spectrum_+(\Delta)>2n-\frac{n}{2}(\sqrt{1+\frac{8}{n}}-1)$).
 \item[(iii)] $(\wt{M},\wt{g})$ (strictly) tangentially stable $\Leftrightarrow$  $(M,g)$ is (strictly) tangentially stable.
 \item[(iv)] $\wt{\Delta}_E$ is bounded below if and only if  $(M,g)$ is physically stable. In this case, $(\wt{M},\wt{g})$ is also physically stable.
\end{itemize}
	\end{thm}
	In particular, Theorem \ref{mainthm:linstability} improves \cite[Theorem 1.4]{Kro18} and closes the gap pointed out in Remark 1.5 in the same paper.
\begin{cor}\label{maincor:stability}
Let $(M^n,g)=(M_1\times M_2,g_1+g_2)$, $n\geq4$ be a product manifold with $\ric_g=(n-1)g$. Then, $(\wt{M},\wt{g})$ is EH unstable and hence also linearly and tangentially unstable by Remark \ref{rem:stability}. If $(M_i,g_i)$ are both linearly stable, $\wt{\Delta}_E$ is unbounded below if and only if $n\leq 8$.
\end{cor}
Our next result considers dynamical stability of sine-cones under the singular Ricci flow mentioned above.
\begin{thm}\label{mainthm:stability}
Let $(M^n,g)$, $n\geq3$ be a closed Einstein manifold with Einstein constant $n-1$, which is strictly tangentially stable. Then, its sine-cone $(\wt{M},\wt{g})$ is dynamically stable under the volume-normalized Ricci-de Turck flow preserving isolated conical singulaties.
\end{thm}
In \cite[Theorem 1.4]{KV18}, the strictly tangentially stable symmetric spaces of compact type were classified. In combination with Theorem \ref{mainthm:stability}, this yields the following examples:
\begin{cor}
	Let $(M^n,g)$, $n\geq 3$ be a closed Einstein manifold with constant $(n-1)$, 
	which is a symmetric space of compact type. If it is a simple Lie group $G$, 
	its sine-cone is dynamically stable under the normalized singular Ricci de Turck flow, if $G$ is one of the following spaces:
	\begin{align}
	\mathrm{Spin}(p)\text{ }(p\geq 6,p\neq 7),\qquad \mathrm{E}_6,
	\qquad\mathrm{E}_7,\qquad\mathrm{E}_8,\qquad \mathrm{F}_4.
	\end{align}
	If $(M,g)$ is a rank-$1$ symmetric space of compact type $G/K$, 
	 its sine-cone is dynamically stable under the normalized singular Ricci de Turck flow, if  $G/K$ is one of the following real Grasmannians
	\begin{equation}
	\begin{aligned}
	&\frac{\mathrm{SO}(2q+2p+1)}{\mathrm{SO}(2q+1)\times \mathrm{SO}(2p)}\text{ }(p\geq 2,q\geq 1),\qquad
	\frac{\mathrm{SO}(8)}{\mathrm{SO}(5)\times\mathrm{SO}(3)},\\
	&\frac{\mathrm{SO}(2p)}{\mathrm{SO}(p)\times \mathrm{SO}(p)}\text{ }(p\geq 4),\qquad
	\frac{\mathrm{SO}(2p+2)}{\mathrm{SO}(p+2)\times \mathrm{SO}(p)}\text{ }(p\geq 4),\\
	&\frac{\mathrm{SO}(2p)}{\mathrm{SO}(2p-q)\times \mathrm{SO}(q)}\text{ }(p-2\geq q\geq 3),
	\end{aligned}
	\end{equation}
	or one of the following spaces:
	\begin{equation}
	\begin{aligned}
	\mathrm{SU}(2p)/\mathrm{SO}(p)\text{ }(n\geq 6),\qquad
	&\mathrm{E}_6/[\mathrm{Sp}(4)/\left\{\pm I\right\}],\qquad \quad
	\mathrm{E}_6/\mathrm{SU}(2)\cdot \mathrm{SU}(6),\\
	\mathrm{E}_7/[\mathrm{SU}(8)/\left\{\pm I\right\}],\qquad&
	\mathrm{E}_7/\mathrm{SO}(12)\cdot\mathrm{SU}(2),\qquad
	\mathrm{E}_8/\mathrm{SO}(16),\\
	\mathrm{E}_8/\mathrm{E}_7\cdot \mathrm{SU}(2),\qquad&
	\mathrm{F}_4/Sp(3)\cdot\mathrm{SU}(2).
	\end{aligned}
	\end{equation}
	\end{cor}

\subsection{Rigidity of sine-cones}
Let $\lambda\in\R$. Then the linearization of the map $\Phi:g\mapsto \ric_g-\lambda\cdot g$ at its zeros is given by
\begin{align*}
d_g\Phi(h)=\frac{1}{2}\Delta_Eh-\delta^*(\delta h+\frac{1}{2}\nabla\trace h)
=\frac{1}{2}\Delta_Eh-\frac{1}{2}L_{\delta h+\frac{1}{2}\nabla\trace h}g,
\end{align*}
see e.g.\ \cite[Theorem 1.174]{Bes08}. In particular, the linearization of the Einstein condition is given by the Einstein operator, up to a constant and a gauge term. If we additionally assume that $\delta h=0$ (i.e.\ $h$ is orthogonal to the orbit of the diffeomorphism group), we can deduce in the smooth case that $h$ is a $TT$-tensor (c.f.  \cite[Theorem 12.30]{Bes08}).
 This motivates the following definition.
\begin{defn}
Let $(M,g)$ be an Einstein manifold. Then an element $h\in\ker(\Delta_E|_{TT})\cap L^2$ is called an infinitesimal Einstein deformation (or IED for short).
\end{defn}
Note that on smooth Einstein manifolds, any IED is globally smooth and thus bounded. On manifolds with isolated conical singularities, an IED is is still smooth by local elliptic regularity. However, it won't necessarily be bounded towards the singularity. Instead, elliptic regularity of weighted spaces (see e.g.\ \cite[Section 9]{Pac13}) shows that $|\nabla^kh|=\mathcal{O}(r^{\alpha -k})$ for some $\alpha>- \frac{n-2}{2}$ and $k\in\N_0$, as $r\to0$. Here, $r$ denotes the distance function from a singular point.

It is well known that for closed Einstein manifolds $(M,g)$, the condition $\ker(\Delta_E|_{TT})=\left\{0\right\}$ implies that $g$ is an isolated point in the moduli space of Einstein metric (c.f.\ the discussion in \cite[Chapter 12]{Bes08}. For noncompact and singular manifolds, this discussion is much more subtle and depends on the behaviour of $h$ at infinity and close to the singularities, respectively. Our next theorem characterizes the existence of IED's on sine-cones.
\begin{thm}\label{mainthm:rigidity}
Let $(M^n,g)$, $n\geq3$ be a closed Einstein manifold with Einstein constant $n-1$ and $(\wt{M},\wt{g})$ be its sine cone. Then the following assertions do hold:
\begin{itemize}
\item[(i)] $(\wt{M},\wt{g})$ admits bounded IED's if and only if $(M,g)$ admits bounded IED's. 
In particular, if $(M,g)$ does not admit bounded IED's, $\wt{g}$ can not be deformed to a family of Einstein metrics $\wt{g}_t$ of the same Einstein constant in such a way that the gauge condition $\delta(\frac{d}{dt}\wt{g}_t|_{t=0})=0$ holds. 
\item[(ii)] If $(M,g)$ is strictly EH-stable, $(\wt{M},\wt{g})$ does not admit ($L^2$-)IED's.
\end{itemize}
	\end{thm}
In particular, part (i) of the theorem asserts that by taking the sine-cone of a suitably scaled positive Einstein manifold, we can not generate new bounded IED's. This is in surprising contrast to the much simpler case of a product of two Einstein manifolds: 
If $(M^n,g)$ is an Einstein manifold with $\ric_g=(n-1)g$, which is isometric to $(M_1\times M_2,g_1+g_2)$, then $(M^n,g)$ carries an IED if one of the two factors satisfies $(M_i,g_i)\neq (S^2,g_{rd})$ and $2(n-1)\in\spectrum(\Delta_0^{M_i})$, see \cite[Proposition 4.8]{Kro15b}.
	\begin{rem}
	The second statement in part (i) of Theorem \ref{mainthm:rigidity} follows from \cite[Theorem 12.30]{Bes08} as explained in the first paragraph of this subsection. The proof of this theorem uses that due to Obata's eigenvalue inequality \cite{Ob62}, the Einstein constant $\lambda$ itself is never in the positive spectrum of the Laplacian. It is not hard to see that Obata's inequality also holds in the conical case. In the case of sine-cones, one can also directly deduce it from Theorem \ref{sinconefunctions}.
	\end{rem}
New IED's which are unbounded but still in $L^2$, can be generated on sine-cones as the following example illustrates.
	\begin{cor}\label{maincor:rigidity}
Let $(M^n,g)=(M_1\times M_2,g_1+g_2)$ be a product manifold with $\ric_g=(n-1)g$. Then, if $n=9$, its sine-cone admits an $L^2$-IED. If the $(M_i,g_i)$ are strictly linearly stable, then $(M^n,g)$ admits IED's if and only if $n=9$. In this case, the space of $L^2$-IED's is one-dimensional.
 \end{cor}
 \vspace{3mm}
\noindent\textbf{Acknowledgements.}
The work of the author is supported through the DFG grant KR 4978/1-1 in the framework of the priority program 2026: \textit{Geometry at infinity}.
\section{Computation of spectra on sine-cones}
Before we start the computations, let us fix some conventions.
For a given Riemannian manifold $(M,g)$, the three manifolds
$(\wt{M},\wt{g}), (\ol{M},\ol{g})$ and $(\wh{M},\wh{g})$ will in this section always be defined as in 
\eqref{eq:sinecone} and \eqref{olMandwhM}.
The operators on the corresponding manifolds will be denoted by $\Delta,\wt{\Delta},\overline{\Delta},\wh{\Delta}$, ect.
We will furthermore use the canonical projections to pull back objects on $M$ to objects on $\widetilde{M},\overline{M}$ or objects on $\widetilde{M},\overline{M}$ to objects on $\wh{M}$. For notational convenience, we will drop the explicit reference to the projections.
For $n\in\N$, we further define two functions by
\begin{align}
\xi_n&:
 x\mapsto -\frac{n-1}{2}+\sqrt{\frac{(n-1)^2}{4}+x},
 \qquad
\eta_n:
 y\mapsto y(y+n-1).
\end{align}
These functions come from the construction of eigenfunctions on $n$-dimensional sphere. If $\Delta_{S^n}v=\lambda\cdot v$, then $r^{\xi_n(\lambda)}v$ is a harmonic function on $\R^{n+1}$. Conversely, if $v\in C^{\infty}(\R^{n+1})$ is a harmonic function which is homogeneous of degree $k$, then $v|_{S^n}$ is an eigenfunction of $\Delta_{S^n}$ with eigenvalue $\eta_n(k)$.

Let us also collect some commutation rules for differential operators on Einstein manifolds.  The divergences of $\omega\in C^{\infty}(T^*M)$ and $h\in C^{\infty}(S^2M)$ are defined with the sign convention such that
\begin{align}
\delta\omega=-g^{ij}\nabla_i\omega_j,\qquad \delta h_k=-g^{ij}\nabla_ih_{jk},
\end{align}
respectively. The formal adjoint $\delta^*:C^{\infty}(T^*M)\to C^{\infty}(S^2M)$ is given by
\begin{align}
(\delta^*\omega)_{ij}=\frac{1}{2}(\nabla_i\omega_j+\nabla_j\omega_i).
\end{align}
Note that $\delta^*$ is related to the Lie derivative by $2\delta^*\omega=L_{\omega^{\sharp}}g$, where $\omega^{\sharp}\in C^{\infty}(TM)$ is the dual vector field of $\omega$ with respect to $g$. If $(M^n,g)$ is Einstein with $\ric_g=\lambda\cdot g$ we have a collection of commutation identities involving these differential operators: For $v\in C^{\infty}(M)$, $\omega\in C^{\infty}(T^*M)$ and $h\in C^{\infty}(S^2M)$, we have
\begin{equation}\label{commutation}
\begin{split}
\Delta_1(dv)&=d(\Delta_0v-\lambda v),\qquad \Delta_0(\delta \omega)=\delta(\Delta_1\omega+\lambda\omega),\\
\Delta_E(\delta^*\omega)&=\delta^*(\Delta_1\omega-\lambda \omega),\qquad \Delta_1(\delta h) =\delta(\Delta_E h+\lambda h),\\
\Delta_E(\nabla^2v)&=\nabla^2(\Delta_0v-2\lambda v),\qquad \Delta_0(\delta \delta h)=\delta\delta(\Delta_Eh+2\lambda h),\\
\Delta_E(v\cdot g)&=(\Delta v-2\lambda v)g,\qquad \Delta(\trace h)=\trace(\Delta_Eh+2\lambda h).
\end{split}
\end{equation}
The computations can be found in \cite{Lic61}, see also \cite[p.\ 8]{Kro15b}. Note that the third line follows directly from the first and the second line and that the formulas on the right hand side follow from the ones on the left hand side by taking the formal adjoints.
\subsection{The Laplace Beltrami operator}
Because the smallest eigenvalue of $\Delta$ is always $0$ we use the notation $\lambda_0=0$ and label the positive eigenvalues of $\Delta$ by $\lambda_1<\lambda_2<\lambda_3\cdots$.
\begin{thm}\label{sinconefunctions}
	Let $(M^n,g)$ be a closed smooth manifold with Laplace spectrum $\left\{\lambda_i\mid i\in \N_0\right\}$. 
	 Then the Laplace spectrum on $(\wt{M},\wt{g})$ is given by the set of real numbers
	\begin{align}
	\wt{\lambda}_{i,j}:=\eta_{n+1}(\xi_n(\lambda_i)+j),\qquad (i,j)\in\N_0\times \N_0,
	\end{align}
	and each of these eigenvalues has multiplicity $\mathrm{mult}(\wt{\lambda}_{i,j})=\mathrm{mult}(\lambda_i)$.
\end{thm}
\begin{rem}
	It can happen that $\wt{\lambda}_{i,j}=\wt{\lambda}_{k,l}$ for $(i,j)\neq(k,l)$. In this case, we count the eigenvalues and their multiplicities in the above theorem separately. The same convention is also used in Theorem \ref{sincone1forms} and Theorem \ref{Einsteinspectrumsincone} below.
\end{rem}
\begin{proof}The construction in this proof is inspired by the construction of eigenfunctions on the round sphere \cite[pp.\ 159-165]{BGM71} and can also seen as a generalization thereof.
	Let $\lambda\in \spectrum(\Delta)$, $k=\xi_n(\lambda)$ and $v\in C^{\infty}(M)$ such that $\Delta v=\lambda\cdot v$. Due to the well-known relation
	\begin{align}\label{polarlaplace}
	\overline{\Delta}=-\partial^2_{rr}-n\cdot r^{-1}\partial_r+r^{-2}\Delta,
	\end{align}
	the function $\ol{v}=r^k\cdot v\in C^{\infty}(\overline{M})$ satisfies $\overline{\Delta}\ol{v}=0$. Now we define some spaces of smooth functions on $\wh{M}$ by
	\begin{align*}
	P_{k,2j}:=\left\{\sum_{l=0}^ja_l r^{k+2l}z^{2(j-l)}v\mid a_l\in\R \right\},\qquad P_{k,2j+1}:=\left\{\sum_{l=0}^ja_l r^{k+2l}z^{2(j-l)+1}v \mid a_l\in\R\right\}.
	\end{align*}
	We now write the Laplacian on $\wh{M}$ as
	\begin{align}
	\wh{\Delta}=-\partial^2_{zz}+\overline{\Delta}=-\partial^2_{zz}-\partial^2_{rr}-n\cdot r^{-1}\partial_r+r^{-2}\Delta.	\end{align}	
	Because $\ol{v}=r^k\cdot v$ is a harmonic function on $\overline{M}$, it is straightforward to show that we have the inclusion $P_{k,0}\oplus P_{k,1}\subset\ker(\wh{\Delta})$ and that
	for each $j\in \N$, $j\geq 2$, the Laplacian satisfies $\wh{\Delta}(P_{k,j})\subset P_{k,j-2}$. Because
	$\dimn P_{k,j}=\dimn P_{k,j-2}+1$, we have $H_{k,j}:=P_{k,j}\cap \kernel\wh{\Delta}\neq\left\{0\right\}$. By the use of the rank theorem, $H_{k,j}$ is one-dimensional for each $j\in \N_0$. Thus we obtain a sequence of harmonic functions $\wh{v}_j\in H_{k,j}$, $j\in\N_0$, which can be written with respect to the variable $s=\sqrt{r^2+z^2}$ as
	\begin{align}
	\wh{v}_j=s^{k+j}\cdot \wt{v}_j,\qquad  \wt{v}_j\in C^{\infty}\cap L^{\infty}(\widetilde{M}).
	\end{align}
	Because $(\wh{M}^{n+2},\wh{g})$ is the cone over $(\widetilde{M}^{n+1},\wt{g})$, we have
		\begin{align}
	\wh{\Delta}=-\partial^2_{ss}-(n+1)\cdot s^{-1}\partial_s+s^{-2}\widetilde{\Delta}.
	\end{align}
	Therefore, $\wt{\Delta}\wt{v}_j=\eta_{n+1}(k+j)\cdot\wt{v}_j$. This shows that all $\wt{\lambda}_{i,j}$ in the statement of the theorem are actually contained in the Laplace spectrum of $(\widetilde{M},\wt{g})$.
		
	To finish the proof, it remains to show that all eigenfunctions on $\widetilde{M}$ are obtained in the above way. 
		By a straightforward calculation, one sees that the intersection
	$ s^2P_{k,j-2}\cap H_{k,j}\subset P_{k,j}$ is trivial so that counting dimensions implies	
	 $P_{k,j}=H_{k,j}\oplus s^2P_{k,j-2}$. Induction yields
	\begin{align*}
	P_{k,2j}=H_{k,2j}\oplus s^2H_{k,2j-2}\oplus\ldots\oplus s^{2j}H_{k,0},\qquad 	P_{k,2j+1}=H_{k,2j+1}\oplus s^2H_{k,2j-1}\oplus\ldots\oplus s^{2j}H_{k,1},
	\end{align*}
    and the $P_{k,j}$ span $P[r,z]\cdot r^kv$ where $P[r,z]$ denotes the polynomials in $r$ and $z$.
	Because the Laplace operator on $\widetilde{M}$ is of the form
	\begin{align}
	\wt{\Delta}=-\partial^2_{\theta\theta}-n\sin(\theta)^{-1}\cos(\theta)\partial_\theta+\sin(\theta)^{-2}\Delta,
	\end{align}
	it preserves spaces of the form 
	\begin{align}V_{\lambda}=\left\{\varphi\cdot v\in L^2(\widetilde{M})\mid \varphi:(0,\pi)\to \R, v\in E(\Delta,\lambda) \right\}.
	\end{align}	
	Because $L^2(M)=\oplus_{i=0}^{\infty}E(\Delta,\lambda_i)$ in the $L^2$-sense,
	\begin{align}
	L^2(\widetilde{M})=\bigoplus_{i=0}^{\infty}V_{\lambda_i}
	\end{align}
	 in the $L^2$-sense as well. Therefore, all eigenfunctions on $\widetilde{M}$ must be contained in $V_{\lambda_i}$ for some $i\in\N_0$. Let $\varphi\in C^{\infty}_{cs}((0,\pi))$, $v\in C^{\infty}(M)$ and extend $\varphi\cdot v$ to a compactly supported smooth function on $\wh{M}$ written as $\psi\cdot r^kv$. Then, $\psi=\psi(r,z)$ can be approximated in the $L^{\infty}$-norm by elements in $P[r,z]$ which are in turn generated by the $H_{k,j}$. Therefore, $\varphi\cdot v$ can be approximated in $L^2$ by sums of the functions $\wt{v}_j$ and by density,
	\begin{align}
	V_{\lambda}\cap L^2(\widetilde{M})=\overline{\mathrm{span}\left\{\wt{v}_j\right\}_{j\in\N_0}}^{L^2}.
	\end{align}
	This argument finishes the proof of the theorem.
%
%
\end{proof}
\begin{rem}\label{conformalkilling}\text{ }
	\begin{itemize}
		\item[(i)] The construction in the proof shows that the eigenspace to the eigenvalue $\wt{\lambda}_{0,1}=n+1$ is spanned by the function $\varphi=\cos(\theta)$. Moreover, ${\gradient}_{\wt{g}}\varphi$ is a conformal Killing field because
		\begin{align}
		L_{{\gradient}_{\wt{g}}\varphi}\wt{g}=2\wt{\nabla}^2\varphi=-2\varphi\wt{g}.
		\end{align}
		\item[(ii)] For a smooth closed Einstein manifold $(M^n,g)$, $n\geq2$ with Einstein constant $n-1$, there are two rigidity statements: On the one hand $\spectrum_+(\Delta)\geq n$ and $n\in \spectrum_+(\Delta)$ only for the standard sphere \cite{Ob62}. On the other hand $(M,g)$ can only admit a nontrivial conformal Killing vector field if it isometric to the round sphere. This also essentially follows from \cite{Ob62}, see also \cite[p.\ 128]{Bes08}. The eigenvalue estimate still holds on $(\wt{M}^{n+1},\wt{g})$ but by (i), both rigidity statements do not extend.
		\item[(iii)] The theorem can be also applied to more general singular spaces $(M,g)$, whose Laplace operator admits a discrete $L^2$-spectrum and the normalized eigenvalues form an orthonormal basis of $L^2(M)$.
In particular, one can compute the Laplace spectra of iterative sequences of sine-cones built on smooth manifolds.
	\end{itemize}
\end{rem}
\subsection{The connection Laplacian on one-forms}
In this subsection, we always assume that $(M^n,g)$ is Einstein with Einstein constant $n-1$. We define
\begin{align}
D(M):=\delta^{-1}(0)\cap C^{\infty}(T^*M).
\end{align}
As is well known, we have the $L^2$-orthogonal splitting
\begin{align}
C^{\infty}(T^*M)=d(C^{\infty}(M))\oplus D(M).
\end{align}
Due to the commutation rules \ref{commutation},
this splitting is preserved by $\Delta_1$ and we have
\begin{align}
\spectrum(\Delta_1)=\spectrum(\Delta_0-(n-1))\cup\spectrum(\Delta_1|_{D(M)}).
\end{align}
Analogous statements do hold on $(\widetilde{M},\wt{g})$. Because the Einstein constant of this manifold is $n$, we get
\begin{equation}\label{bla}
\spectrum(\wt{\Delta}_1)=\spectrum(\wt{\Delta}_0-n)\cup\spectrum(\Delta_1|_{D(\widetilde{M})})
\end{equation}
in this case.
\begin{thm}\label{sincone1forms}
Let $(M^n,g)$, $n\geq2$ be a smooth closed Einstein manifold with constant $n-1$ and let
\begin{align}
\spectrum(\Delta_0)=\left\{\lambda_i\mid i\in\N_0\right\},\qquad
\spectrum(\Delta_1|_{D(M)})=\left\{\mu_i\mid i\in\N\right\}.
\end{align}
%
	Then the spectrum of $\wt{\Delta}_1$ restricted to $d(C^{\infty}(\wt{M})$ is given by
\begin{align*}
\left\{\wt{\lambda}^{(1)}_{i,j}=\eta_{n+1}(\xi_n(\lambda_i)+j)-n
	\mid (i,j)\in \N_0\times\N_0\setminus\left\{(0,0)\right\}\right\},
\end{align*}	
	with multiplicities $\mult(\wt{\lambda}^{(1)}_{i,j})=\mult(\lambda_i)$.
Then the spectrum of $\wt{\Delta}_1$ restricted to $D(\wt{M})$	is given by the union of the following sets:
	\begin{equation}
	\begin{split}
	&\left\{\wt{\lambda}^{(2)}_{i,j}=\eta_{n+1}(\xi_n(\lambda_i)+j)-1\mid (i,j)\in\N\times\N_0\right\},\\
	&\left\{\wt{\mu}_{i,j}=\eta_{n+1}(\xi_n(\mu_i+1)+j)-1\mid (i,j)\in\N\times\N_0\right\},
	\end{split}
	\end{equation}
	with multiplicities $\mult(\wt{\lambda}^{(2)}_{i,j})=\mult(\lambda_i)$ and $\mult(\wt{\mu}_{i,j})=\mult(\mu_{i})$.
\end{thm}
\begin{rem}\text{}
	\begin{itemize}
			\item[(i)]If $(M,g)$ is a circle, then the construction of the proof still works but the $\wt{\lambda}^{(k)}_{i,j}$ in the theorem can become negative, if $\lambda$ is small enough (i.e.\ if the circle is very large). This happens because due to bad blowup behaviour at the conical singularity, the gradient of the corresponding eigenfunction is not in $L^2$. For circles with small enough length so that all the above values are positive, the theorem still holds in this form.
	\item[(ii)] If $(M,g)$ is a singular space, this theorem may still hold but this also depends on the regularity of eigenfunctions and eigenforms at the singularities. It holds for iterative sequences of sine-cones constructed from a smooth manifold.
	\end{itemize}
\end{rem}
Because of \eqref{bla} and Theorem \ref{sinconefunctions}, the spectrum of $\wt{\Delta}_1$ restricted to $d(C^{\infty}(\widetilde{M}))$ is given by the $\wt{\lambda}^{(1)}_{i,j}$.
To compute the spectrum of $\wt{\Delta}_1$ on $D(\widetilde{M})$, we introduce the spaces 
\begin{align*}
V_{\mu}&:=\left\{\varphi\cdot \omega \in L^2(D(\widetilde{M})) \mid \varphi:(0,\pi)\to\R,\quad \omega\in E(\Delta_1|_{D(M)},\mu)\right\},\\
W_{\lambda}&:=\left\{\varphi\cdot dv+\psi\cdot v\cdot d\theta\in L^2(D(\widetilde{M})) \mid \varphi,\psi:(0,\pi)\to\R, v\in E(\Delta_0,\lambda) \right\}.
\end{align*}
\begin{lem}\label{Wlambdaempty}
If $\lambda=0$, then $W_{\lambda}=0$.
\end{lem}
\begin{proof}
	Note that $\lambda=0$ implies that $v$ is constant, hence $dv=0$. Therefore, any element of $W_0$ is of the form $\psi d\theta$. By \eqref{tildelemma0}, 
	\begin{align}0=\wt{\delta}(\psi d\theta)=-\partial_{\theta}\psi-n\frac{\cos(\theta)}{\sin(\theta)}\psi,
	\end{align}
	so that $\psi$ is necessarily of the form $\psi= a\cdot \sin^{-n}$ for some $a\in \R$. This is a contradiction to the assumption $\psi\cdot d\theta\in L^2(D(\widetilde{M}))$.
\end{proof}
\noindent In \cite[Section 2]{Kro15c}, we have seen that $\wt{\Delta}_1$ preserves the spaces $V_{\mu}$ and $W_{\lambda}$ and that 
\begin{align}
L^2(D(\widetilde{M}))=\left(\bigoplus_{i=1}^{\infty}V_{\mu_i}\right)\bigoplus\left(\bigoplus_{i=1}^{\infty}W_{\lambda_i}\right)
\end{align}
in the $L^2$-sense. The statement of the theorem thus follows from Lemma \ref{Vlambdaoneforms} and Lemma \ref{Wmuoneforms} below.
\begin{lem}\label{Vlambdaoneforms}
	The spectrum of $\wt{\Delta}_1$ restricted to $V_{\mu}$ is given by 
	\begin{align}
	\left\{\eta_{n+1}(\xi_n(\mu)+1)-1\mid j\in \N_0\right\}
	\end{align}
	and each of these eigenvalues has the same multiplicity as $\mu$ as an eigenvalue of $\Delta_1$.	
\end{lem}
\begin{proof} Let $\ol{\omega}\in C^{\infty}(T^*\overline{M})$ be of the form $\ol{\omega}=\varphi\cdot r\omega$ where  ${\varphi}\in C^{\infty}(0,\infty)$. We then have
\begin{align}
\ol{\delta}\ol{\omega}=0,\qquad \ol{\omega}(\partial_r)=0.
\end{align}
By Lemma \ref{hatdiv}, 
\begin{align}
\ol{\Delta}_1\ol{\omega}=r(-\partial^2_{rr}{\varphi}\cdot\omega-nr^{-1}\partial_r{\varphi}\cdot\omega+{\varphi}\cdot r^{-2}(\Delta_1+1)\omega).
\end{align}
Thus if $\Delta_1\omega=\mu\cdot\omega$, $\ol{\omega}$ is harmonic if $\varphi(r)=r^l$ and $l=\xi_n(\mu+1)$.
Now consider some spaces of one-forms on $\wh{M}$, given by
\begin{align*}
P_{l,2j}:=\left\{\sum_{m=0}^ja_m r^{l+2m}z^{2(j-m)}\ol{\omega}\mid a_m\in\R \right\},\quad P_{l,2j+1}:=\left\{\sum_{m=0}^ja_m r^{l+2m}z^{2(j-m)+1}\ol{\omega} \mid a_m\in\R\right\}.
\end{align*}
It is straightforward to show that $P_{l,0}\oplus P_{l,1}\subset\ker(\wh{\Delta})$ and that
for each $j\in \N$, $j\geq 2$, the Laplacian satisfies $\wh{\Delta}_1(P_{l,j})\subset P_{l,j-2}$. Because
$\dimn P_{l,j}=\dimn P_{l,j-2}+1$, $H_{l,j}:=P_{l,j}\cap \kernel\wh{\Delta}_1\neq\left\{0\right\}$ and it is also not hard to see that $H_{l,j}$ is one-dimensional for each $j\in \N_0$. Thus we obtain a sequence of harmonic forms $\wh{\omega}_j\in  D(\wh{M})$, $j\in\N_0$, which can be written with respect to the variable $s=\sqrt{r^2+z^2}$ as
\begin{align}
\wh{\omega}_j=s^{l+j} (s\cdot \wt{\omega}_j),\qquad  \wt{\omega}_j\in C^{\infty}\cap L^{\infty}(D(\widetilde{M})).
\end{align}
Let now $\wh{\omega}\in C^{\infty}(T^*\wh{M})$ be of the form $\wh{\omega}=\varphi\cdot s\wt{\omega}$, $\varphi=\varphi(s)\in C^{\infty}((0,\infty))$.
By Lemma \ref{hatdiv} again, we have
\begin{align}
\wh{\Delta}_1\wh{\omega}=s(-\partial^2_{ss}{\varphi}\cdot\wt{\omega}-(n+1)\partial_s{\varphi}\cdot\wt{\omega}+{\varphi}\cdot s^{-2}(\wt{\Delta}_1+1)\wt{\omega})
\end{align}
for any divergence-free one-form $\wt{\omega}$ on $\widetilde{M}$ so that
we get $\wt{\Delta}_1\wt{\omega}_j=(\eta_{n+1}(\xi_n(\mu+1)+j)-1)\cdot\wt{\omega}_j$. A density argument completely analogous to the case of
smooth functions shows that all the eigenforms of the connection Laplacian on $V_{\mu}$ are of this form.
\end{proof}
\begin{lem}\label{Wmuoneforms}
	The spectrum of $\wt{\Delta}_1$ restricted to $W_{\lambda}$ is given by 
	\begin{align}\left\{\eta_{n+1}(\xi_n(\lambda)+j)-1\mid
	j\in \N_0\right\},
	\end{align}
	and each of these eigenvalues has the same multiplicity as $\lambda$ as an eigenvalue of $\Delta_0$.	
\end{lem}
\begin{proof}
Suppose that $\wt{\omega}\in W_{\lambda}$ satisfies $\Delta_1\wt{\omega}=\nu\wt{\omega}$ for some $\nu\geq -\frac{n^2}{4}$. Let $\wh{\omega}=s\cdot \wt{\omega}\in C^{\infty}(T^*\wh{M})$. Then by Lemma \ref{hatdiv},
\begin{align}\label{conditionshatomega}
\wh{\delta}(s^k\wh{\omega})=0,\qquad s^k\wh{\omega}({\partial_s})=0,\qquad \wh{\Delta}_1(s^k\wh{\omega})=0
\end{align}
for $k:=\xi_{n+1}(\mu+1)$. Therefore, $\nu$ appears as an eigenvalue of $\wt{\Delta}_1$ restricted to $W_{\lambda}$ if and only if we find a form $s^k\wh{\omega}$ such that the conditions \eqref{conditionshatomega} hold.
To do so, we make the following ansatz, where we use the coordinates $r,z$ instead of $s,\theta$ on $\wh{M}$:
\begin{align}\label{formansatz}
s^k\wh{\omega}=P[r,z]\cdot v\cdot dz+Q[r,z]\cdot v\cdot dr+R[r,z]\cdot r\cdot dv
\end{align}
It is clear that $s^k\wh{\omega}$ must be of that form because it restricts to an element in $W_{\lambda}$.
Demanding $s^k\wh{\omega}(\partial_s)=0$ is equivalent to $s^k\wh{\omega}(r\partial_r+z\partial_z)=0$ which in turn is equivalent to
\begin{align}\label{equationQ}
Q=-\frac{z}{r}\cdot P.
\end{align}
Lemma \ref{coupled1forms} in the appendix implies that the condition $\wh{\delta}(s^k\wh{\omega})=0$ is equivalent to the equation
\begin{align}\label{equationR}
\lambda\cdot R=r\partial_z P+r\partial_r Q+n\cdot Q.
\end{align}
In the following, we want to use \eqref{formansatz} to translate the equation $\wh{\Delta}_1(s^k\wh{\omega})=0$ to a system of equations on $P,Q$ and $R$.
Because $dz$ is parallel,
\begin{align}\label{delta1P}
\wh{\Delta}_1(P\cdot v\cdot dz)=\wh{\Delta}_0(P\cdot v)\cdot dz=(\wh{\Delta}_0P\cdot v+\lambda r^{-2}P\cdot v)\cdot dz.
\end{align}
By Lemma \ref{gradhessdr},
\begin{align}
\ol{\nabla}dr=r\cdot g,\qquad \ol{\Delta}_1dr=nr^{-2}dr,
\end{align}
and therefore,
\begin{align}
\ol{\nabla}(vdr)=dv\otimes dr+vr\cdot g
\end{align}
and
\begin{equation}
\begin{split}
\ol{\Delta}_1(vdr)&=\ol{\Delta}_0v\cdot dr+v\cdot \ol{\Delta}_1dr-2\langle dv,\ol{\nabla}dr\rangle_{\ol{g}}=r^{-2}(\lambda+n)vdr-2r^{-2}(rdv).
\end{split}
\end{equation}
Therefore,
\begin{align}\label{delta1Q}
\wh{\Delta}_1(Q\cdot vdr)=\wh{\Delta}_0Q\cdot vdr+Q\ol{\Delta}_1(vdr)
=\wh{\Delta}_0Q\cdot vdr+Qr^{-2}((\lambda+n)vdr-2rdv).
\end{align}
By Lemma \ref{gradhessdr} again, we have
\begin{align}
\ol{\nabla}(rdv)=r\nabla dv-dv\otimes dr,\qquad \ol{\Delta}_1(rdv)=r^{-2}(r(\Delta_1+1)dv-2\Delta_0 v\cdot dr),
\end{align}
which yields
\begin{equation}\begin{split}\label{delta1R}
\wh{\Delta}_1(R\cdot r\cdot dv)&=(\wh{\Delta}_0R) r\cdot dv+Rr^{-2}(r(\Delta_1+1)dv-2\Delta_0 v\cdot dr)\\
&=(\wh{\Delta}_0R) r\cdot dv+Rr^{-2}(r(\lambda-n+2)dv-2\lambda v\cdot dr.
\end{split}
\end{equation}
In the last equality, we used one commutation identity of \eqref{commutation}. By summing up \eqref{delta1P}, \eqref{delta1Q} and \eqref{delta1R}, we obtain that the condition $\wh{\Delta}_1(s^k\wh{\omega})=0$ is equivalent to the system of equations
\begin{equation}\begin{split}\label{harmonicityQR}
\wh{\Delta}P+r^{-2}\lambda P&=0,\\
\wh{\Delta}Q+r^{-2}(\lambda+n)Q-2r^{-2}\lambda R&=0,\\ \wh{\Delta}R+r^{-2}(\lambda+2-n)R-2r^{-2}Q&=0.
\end{split}
\end{equation}
Suppose now that $P$ satisfies the first equation in \eqref{harmonicityQR} and $Q$ and $R$ are defined by \eqref{equationQ} and \eqref{equationR}.
Then by Lemma \ref{formulas1} in the appendix, $P,Q,R$ satisfy \eqref{harmonicityQR} so that $\wh{\Delta}_1(s^k\wh{\omega})=0$. 
Thus it remains to find solutions for the equation on $P$, i.e.\ to find $P$ such that $P\cdot v$ is a harmonic function on $\wh{M}$. Because $|s^k\wh{\omega}|_{\wh{g}}=s^k|\wt{\omega}|_{\wt{g}}$ we find that
\begin{align}\label{scalep}
P[\alpha r,\alpha z]=\alpha^kP[r,z],\qquad
Q[\alpha r,\alpha z]=\alpha^kQ[r,z]\qquad
R[\alpha r,\alpha z]=\alpha^kR[r,z].
\end{align}
In the proof of Theorem \ref{sinconefunctions}, we have seen that for each
\begin{align}
 k\in \left\{\xi_{n}(\lambda)+j\mid
  j\in \N_0\right\},
\end{align}
we find a one-dimensional space of functions $P_j=P_j[r,z]$ satisfying \eqref{scalep} and  $\wh{\Delta}(P_j\cdot v)=0$. Each such $P_j$ defines $Q_j$ and $R_j$ by \eqref{equationQ} and \eqref{equationR} such that the triple $(P_j,Q_j,R_j)$ satisfies \eqref{harmonicityQR} and \eqref{scalep}. Thus for each $k=\xi_{n}(\lambda)+j$, we have a $1$-form 
\begin{align}
s^k\wh{\omega}_j=P_j\cdot dz+Q_j\cdot v\cdot dr+R_j\cdot r\cdot dv,
\end{align}
which fulfills the conditions in \eqref{conditionshatomega} and by restricting to $\widetilde{M}$, it defines $\wt{\omega}_j\in W_{\lambda}$ such that
\begin{align}\widetilde{\Delta}_1\wt{\omega}_j=(\eta_{n+1}(\xi_{n}(\lambda)+j)-1)\wt{\omega}_j.
\end{align}
Similarly as in the proof of Theorem \ref{sinconefunctions}, one shows that the span of the $P_j$ is dense in $C^{\infty}(\R_+\times \R)$ with respect to uniform convergence.
Because $P_j$ uniquely determines $Q_j$ and $R_j$, the triples $(P_j,Q_j,R_j)$ are dense in the solutions of the system \eqref{harmonicityQR} with respect to uniform converence.
Therefore, the $\widetilde{\omega}_j$ are dense in $W_{\lambda}$ with respect to uniform convergence, hence in $L^2$. This finishes the proof of the lemma.
%
%
\end{proof}
\begin{rem}\label{Killing}
A straightforward computation shows that $2\delta\circ\delta^*+d\circ\delta=\Delta_1\omega-(n-1)\omega$, since $(M,g)$ is Einstein with constant $n-1$.
Therefore, $\mu_i\geq n-1$ (with the notation of Theorem \ref{sincone1forms}) and  $\Delta_1\omega=(n-1)\omega$ if and only if  $\delta^*\omega=0$, i.\,e. $\omega^{\sharp}$ is a Killing vector field.
An analogous assertion holds for $(\wt{M},\wt{g})$ with $n-1$ replaced by $n$.

If $\mu_1=(n-1)$, then $\wt{\mu}_{1,0}=n$ and the corresponding $1$-forms on $\widetilde{M}$ are duals of Killing vector fields as well.
If $\lambda_1=n$ (i.e.\ the gradient of the corresponding eigenfunction is a conformal Killing vector field, see Remark \ref{conformalkilling}), then $\wt{\lambda}^{(2)}_{1,0}=n$ and the corresponding $1$-forms on $\widetilde{M}$ are also duals of Killing vector fields. Therefore, we obtain
\begin{align}
\ker(\wt{\delta}^*)=E(\wt{\Delta}_1,n)=E(\Delta_1,n-1)\oplus E(\Delta_0,n)=\ker(\delta^*)\oplus E(\Delta_0,n).
\end{align}
\end{rem}
\subsection{The Einstein operator}
In this section, we again assume that $(M^n,g)$ is Einstein with constant $n-1$. In addition, we assume that the dimension of $M$ is $n\geq3$.
 We have the $L^2$-orthogonal splitting
\begin{align}\label{eq:decomp}
C^{\infty}(S^2M)=C^{\infty}(M)\cdot g\oplus \left\{n\nabla^2v+\Delta v\cdot g\mid v\in C^{\infty}(M) \right\}\oplus \delta^*(D(M))\oplus TT(M),
\end{align}
where $TT(M)=\left\{h\in C^{\infty}(S^2M)\mid \trace h=0,\delta h=0\right\}$ denotes the space of transverse traceless tensors on $h$.
The Einstein operator has a block diagonal form with respect to this decomposition \cite[p. 130]{Bes08} (see  also \cite[Section 2]{Kro15c} for the refined version stated here). An analogous decomposition holds for its sine-cone $(\wt{M},\wt{g})$ although it admits conical singularities. In this case, \ref{eq:decomp} holds if we replace $C^{\infty}$ in each factor on both sides by the $L^2$-closure of $C^{\infty}_{cs}$, see \cite[Section 2]{Kro18} for details.
\begin{lem}\label{unbounded}
	Suppose that $(M,g)$ admits a $TT$-eigentensor $h$ of $\Delta_E$ with eigenvalue $\mu<-\frac{(n-1)^2}{4}$. Then, $\wt{\Delta}_E$ is unbounded from below. 
	\end{lem}
\begin{proof}
Let $\wt{h}$ be of the form $\wt{h}=\varphi \sin^2 h$. Then, we know from \cite[Section 2]{Kro15c} that
\begin{align*}
(\wt{\Delta}_E\wt{h},\wt{h})_{L^2(\wt{g})}=\int_0^{\pi}(\dot{\varphi})^2\sin^nds+\mu\int_0^{\pi}\varphi^2\sin^{n-2}ds,\qquad \left\|\wt{h}\right\|_{L^2{(\wt{g})}}^2=\int_0^{\pi}\varphi^2\sin^nds,
\end{align*}
where we assumed that $\left\|h\right\|_{L^2(g)}=1$ for simplicity.	
We know by the Hardy inequality that
\begin{align*}
\inf_{\varphi\in C^{\infty}_{cs}((0,\infty))\setminus\left\{0\right\}}\frac{\int_{0}^{\pi}(\dot{\varphi})^2s^nds}{\int_0^{\pi}\varphi^2s^{n-2}ds}=\frac{(n-1)^2}{4}.
\end{align*}
Let $\epsilon_i$ be a zero sequence and 
$\varphi_i\in C^{\infty}_{cs}((0,\infty))$ be a sequence 
such that
\begin{align*}
\frac{\int_{0}^{\pi}(\dot{\varphi_i})^2s^nds}{\int_0^{\pi}\varphi_i^2s^{n-2}ds}\leq\frac{(n-1)^2}{4}+\epsilon_i.
\end{align*}
Because the quotient above is invariant under scaling $s\mapsto\alpha s$, we may assume that $\supp(\varphi_i)\subset (0,\epsilon_i)$. We have $(1-\delta_i)s\leq \sin(s)\leq s$ for $ s\in (0,\epsilon_i)$ and a zero sequence $\delta_i$. Therefore,
\begin{align*}
\int_0^{\pi}(\dot{\varphi_i})^2\sin^nds+&\mu\int_0^{\pi}\varphi_i^2\sin^{n-2}ds\leq \int_0^{\pi}(\dot{\varphi}_i^2)s^nds+\mu(1-\delta_i)^{n-2}\int_0^{\pi}\varphi_i^2s^{n-2}ds\\
&\leq\left(\frac{(n-1)^2}{4}+\mu+\epsilon_i-(1-(1-\delta_i)^{n-2})\mu\right)\int_0^{\pi}\varphi_i^2s^{n-2}ds\\
&\leq -\epsilon\int_0^{\pi}\varphi_i^2s^{n-2}ds
\end{align*}
for some $\epsilon>0$ and $i$ large enough. Consequently,
\begin{align*}
\frac{\int_0^{\pi}(\dot{\varphi_i})^2\sin^nds+\mu\int_0^{\pi}\varphi_i^2\sin^{n-2}ds}{\int_0^{\pi}\varphi_i^2\sin^{n}ds}
\leq -\epsilon \frac{\int_0^{\epsilon_i}\varphi_i^2s^{n-2}ds}{\int_0^{\epsilon_i}\varphi_i^2s^{n}ds}\leq -\epsilon\cdot \epsilon_i^{-2}\to-\infty
\end{align*}
as $i\to\infty$
which proves the lemma.
\end{proof}

\begin{thm}\label{Einsteinspectrumsincone}Let $(M^n,g)$, $n\geq3$ be a smooth closed Einstein manifold of constant $n-1$.
Let
\begin{equation}\begin{split}
\spectrum(\Delta_0)&=\left\{\lambda_i\mid i\in\N_0\right\},\qquad
\spectrum(\Delta_1|_{D(M)})=\left\{\mu_i\mid i\in\N\right\},\\
\spectrum(\Delta_E|_{TT(M)})&=\left\{\kappa_i\mid i\in\N\right\},
\end{split}
\end{equation}
 and suppose that $\kappa_i\geq-\frac{(n-1)^2}{4}$ for all $i\in\N$.
	\begin{itemize}
		\item[(i)]	Suppose that $\lambda_1>n$ and $\mu_1>n-1$. Then the spectrum of $\wt{\Delta}_E$ restricted to
	 \begin{align}C^{\infty}(\widetilde{M})\cdot\wt{g}\oplus \left\{n\wt{\nabla}^2v+\wt{\Delta} v\cdot \wt{g}\mid v\in C^{\infty}(\widetilde{M}) \right\}
	 \end{align}
	  is given by
		\begin{align}
	\wt{\lambda}_{i,j}^{(1)}:=\eta_{n+1}(\xi_n(\lambda_i)+j)-2n,\qquad (i,j)\in\N_0\times\N_0,
	\end{align}
	and each of these eigenvalues has multiplicity 
	\begin{align}\mathrm{mult}(\lambda_{i,j}^{(1)})=2\mathrm{mult}(\lambda_i)\quad\text{ if }(i,j)\in \N_0\times\N_0\setminus{\left\{(0,0),(0,1)\right\}},\end{align}
	 and $\mathrm{mult}(\lambda_{0,i}^{(1)})=\mult(\lambda_0)=1$ if $i=0,1$. The spectrum of $\wt{\Delta}_E$ restricted to $\wt{\delta}^*(D(\widetilde{M}))$ is given by
	\begin{align}
 \label{lambda2} \wt{\lambda}^{(2)}_{i,j}&=\eta_{n+1}(\xi_n(\lambda)+j)-1-n,\qquad(i,j)\in \N\times\N_0,\\
\label{mu1}\wt{\mu}^{(1)}_{i,j}&=\eta_{n+1}(\xi_n(\mu_i+1)+j)-1-n,\qquad (i,j)\in \N\times\N_0,
	\end{align}
	with $\mult(\wt{\lambda}^{(2)}_{i,j})=\mult(\lambda_i)$ and $\mult(\wt{\mu}^{(1)}_{i,j})=\mult(\mu_i)$.
	Finally, the spectrum of $\wt{\Delta}_E$ restricted to $TT(\widetilde{M})$ is given by
	\begin{align}
 \label{lambda3}	\wt{\lambda}_{i,j}^{(3)}&=\eta_{n+1}(\xi_n(\lambda_i)+j),\qquad (i,j)\in \N\times\N_0,	\\
\label{mu2}	\wt{\mu}_{i,j}^{(2)}&=\eta_{n+1}(\xi_n(\mu_i+1)+j),\qquad (i,j)\in \N\times\N_0,\\
	\wt{\kappa}_{i,j}^{(1)}&=\eta_{n+1}(\xi_n(\kappa_i)+j),\qquad (i,j)\in \N\times\N_0,		
	\end{align}
	with $\mult(\wt{\lambda}_{i,j}^{(3)})=\mult(\lambda_i)$, $\mult(	\wt{\mu}_{i,j}^{(2)})=\mult(\mu_i)$ and $\mult(\wt{\kappa}_{i,j}^{(2)})=\mult(\kappa_i)$.
\item[(ii)] If $\lambda_1=n$, then (i) holds up to the following changes:
\begin{itemize}\item[(a)] $\mathrm{mult}(\wt{\lambda}_{1,0}^{(1)})=\mathrm{mult}(\lambda_1)$
	           \item[(b)] The value $\tilde{\lambda}_{1,0}^{(2)}$ does not appear in \eqref{lambda2}
	           \item[(c)] The values $\tilde{\lambda}_{1,j}^{(3)}$, $j\in\N_0$ do not appear in \eqref{lambda3}
\end{itemize}
\item[(iii)] If $\mu_1=n-1$, then (i) holds up to the following changes:
\begin{itemize}
	\item[(a)] The value $\tilde{\mu}_{1,0}^{(1)}$ does not appear in \eqref{mu1}
	\item[(b)] The values $\tilde{\mu}_{1,j}^{(2)}$, $j\in\N_0$ do not appear in \eqref{mu2}
\end{itemize}
 \end{itemize}
\end{thm}
\begin{rem}
	The theorem may still hold for singular manifolds and in dimensions $n=1,2$, provided that 
	 we have $dv,\nabla^2v,\delta^*\omega\in L^2$. It will hold for iterative sequences of $\sin$-cones builded up on smooth Einstein manifolds. In low dimensions, $D(M)$ and $TT(M)$ may be finite dimensional in low dimensions which yields a slight technical simplification.
\end{rem}
\noindent By the commutation identities in \eqref{commutation}, it remains to compute the spectrum of $\wt{\Delta}_E$ on $TT$-tensors, since the other parts of the spectrum have already been computed in Theorem \ref{sinconefunctions} and Theorem \ref{sincone1forms}. The cases (iia) and (iib) and (iiia) follow from Remark \ref{Killing} and the fact that the kernel of the map
\begin{align}
\Psi:C^{\infty}(M)\to C^{\infty}(S^2M),\qquad v\mapsto \nabla^2v+\Delta v\cdot g
\end{align}
is given by $E(\Delta ,0)\oplus E(\Delta,n)$, see \cite[Lemma 2.4]{Kro15c}.
To compute the spectrum of $\wt{\Delta}_E$ on $TT$-tensors, we first introduce the notations
\begin{align*}\wt{h}^{(1)}&=\varphi\cdot \sin(\theta)^2h,\\
\wt{h}^{(2)}&=\varphi\cdot \sin(\theta)^2\delta^*\omega+\psi\cdot \sin(\theta)\omega\odot d\theta,\\
\wt{h}^{(3)}&=\varphi\cdot \sin(\theta)^2(\nabla^2v+\Delta v\cdot g)+\psi\cdot \sin(\theta)dv\odot d\theta+\chi v(nd\theta\otimes d\theta-\sin(\theta)^2g),
\end{align*}
with $h\in TT(M)$, $w\in D(M)$ and $v\in C^{\infty}(M)$, and the spaces
\begin{align*}
V_{\kappa}&=\left\{\wt{h}^{(1)} \in L^2(TT(\widetilde{M})) \mid \varphi:(0,\pi)\to\R,\quad h\in E(\Delta_E|_{TT(M)},\kappa) \right\},\\
W_{\mu}&=\left\{\wt{h}^{(2)}\in L^2(TT(\widetilde{M})) \mid \varphi,\psi:(0,\pi)\to\R,\quad \omega \in E(\Delta_1|_{D(M)},\mu) \right\},\\
X_{\lambda}&=\left\{\wt{h}^{(3)}\in L^2(TT(\widetilde{M})) \mid \varphi,\psi,\chi:(0,\pi)\to\R,\quad v \in E(\Delta_0,\lambda) \right\}.
\end{align*}
Here, the symmetric tensor product of $\omega,\eta\in C^{\infty}(T^*M)$ is defined with the convention that $\omega\odot\eta=\omega\otimes\eta+\eta\otimes\omega\in C^{\infty}(S^2M)$. 
\begin{lem}The following of the above spaces are trivial:
	\begin{itemize}
		\item[(i)] If $\mu=n-1$, then $W_{\mu}=0$.
		\item[(ii)] If $\lambda=0$, then $V_{\lambda}=0$.
		\item[(iii)] If $\lambda=n$, then $V_{\lambda}=0$.
	\end{itemize}
\end{lem}
\begin{proof}
		Recall from Remark \ref{Killing} that if $\mu=n-1$, $\omega^{\sharp}$ is a Killing field, hence $\delta^*\omega=0$. Therefore, any element of $W_{\mu}$ is of the form $\psi\cdot \sin(\theta)\omega\odot d\theta$. By \eqref{tildelemma1}, 
\begin{align}0=\wt{\delta}(\psi\cdot \sin(\theta)\omega\odot d\theta)=-[\partial_{\theta}\psi\cdot\sin+(n+1)\cos\cdot\psi]\omega,
\end{align}
so that $\psi$ is necessarily of the form $\psi= a\cdot \sin^{-(n+1)}$ for some $a\in \R$. This is a contradiction to the assumption $\psi\cdot \sin(\theta)\omega\odot d\theta\in L^2(TT(\widetilde{M}))$. This proves (i).\\
To prove (ii), we remark that $v$ is a constant function and therefore,
 any element of $V_{\lambda}$ is of the form $\chi v(nd\theta\otimes d\theta-\sin^2g)$. By \eqref{tildelemma3}, we have
 \begin{align}
 0&=\wt{\delta}(\chi v(nd\theta\otimes d\theta-\sin^2g))=\chi\cdot dv-n[\partial_{\theta}\chi+(n+1)\frac{\cos}{\sin}\chi]vd\theta
 ,
 \end{align}
 so that $\chi$ does necessarily vanish.\\
 If $v\in C^{\infty}(M)$ is such that $\Delta v=n\cdot v$, then $\nabla^2v+\Delta v\cdot g=0$. Therefore in this case, any element $\tilde{h}\in V_{\lambda}$ is of the form
 \begin{align}\label{elementinVlambda}
 \wt{h}=\psi\cdot \sin dv\odot d\theta+\chi v(nd\theta\otimes d\theta-\sin^2g),
 \end{align}
 and due to \eqref{tildelemma2},
 \begin{align}
 \wt{\delta}(\psi\sin dv\odot d\theta)=n\sin^{-1}
\psi v\cdot d\theta-[(n+1)\psi\cos+\partial_{\theta}\psi\cdot\sin ]dv.
\end{align}
Thus, we get
\begin{equation}
\begin{split}
0&=\wt{\delta}\wt{h}
=n[\sin^{-1}\psi -\partial_{\theta}\chi-(n+1)\frac{\cos}{\sin}\chi]v\cdot d\theta+[\chi-(n+1)\cos\cdot \psi
-\sin\cdot\partial_{\theta}\psi]dv,
\end{split}
\end{equation} which implies that the functions $\psi,\chi:(0,\pi)\to \R$ have to satisfy the coupled system
\begin{equation}\label{ODE}
\begin{split}
\partial_{\theta}\chi&=-(n+1)\frac{\cos}{\sin}\chi+\sin^{-1}\psi,\\
\partial_{\theta}\psi&=-(n+1)\frac{\cos}{\sin}\cdot \psi+\sin^{-1}\chi.
\end{split}
\end{equation}
For a solution $(\chi,\psi)$ of this system, we define its energy as $E=\chi^2+\psi^2$.
Let $\epsilon>0$ be arbitrary and choose $\delta>0$ so small that $(1+\epsilon)\cos(\theta)<-1$ for $\theta\in (\pi-\delta,\pi)$.
On this subinterval, the energy satisfies the differential inequality
\begin{equation}
\begin{split}
\partial_{\theta}E&=-2(n+1)\frac{\cos}{\sin}\chi^2-2(n+1)\frac{\cos}{\sin}\psi^2+4\sin^{-1}\chi\psi \\
&\geq -2(n+1)\frac{\cos}{\sin}(\chi^2+\psi^2)-2\sin^{-1}(\chi^2+\psi^2)\\
&\geq -2(n-\epsilon)\frac{\cos}{\sin}E,
\end{split}
\end{equation}
and thus, $E\geq C_1\cdot \sin^{-2(n-1-\epsilon)}$ for some $C_1>0$ on  $ (\pi-\delta,\pi)$ unless $E$ and therefore the pair $(\chi,\psi)$ is identically zero. Because space of solutions of \eqref{ODE} is invariant under the involution 
 \begin{align}
 (\chi(\theta),\psi(\theta))\to (\chi(\pi-\theta),-\psi(\pi-\theta)),
 \end{align}
an analogous estimate also holds on $ (0,\delta)$. However, the pointwise norm of $\tilde{h}$ is
\begin{align}
|\wt{h}|_{\wt{g}}^2=2\psi^2|dv|_g^2+2n^2\chi^2v^2,
\end{align}
and by integrating
\begin{align}
\left\|\wt{h}\right\|_{L^2(\wt{g})}^2=\int_0^{\pi}\int_M|\wt{h}|_{\wt{g}}^2\sin^n\dv_gd\theta
\geq 2n\left\|v\right\|_{L^2(g)}^2\int_0^{\pi}E\sin^nd\theta,
\end{align}
but due to the estimate on $E$ above, the integral on the right hand side is infinite for $n\geq3$ unless $\wt{h}\equiv0$.
This finishes the proof of (iii).
\end{proof}

The lemma above explains the cases (iic) and (iiib) in Theorem \ref{Einsteinspectrumsincone}.
 Therefore, we may assume $\lambda>n$  and $\mu_>n-1$  from now on.
We have seen in \cite[Section 2]{Kro15c} that $\wt{\Delta}_E$ preserves these spaces and that
\begin{align}
L^2(TT(\widetilde{M}))=\left(\bigoplus_{i=1}^{\infty}V_{\kappa_i}\right)\bigoplus\left(\bigoplus_{i=1}^{\infty}W_{\mu_i}\right)\bigoplus\left(\bigoplus_{i=1}^{\infty}X_{\lambda_i}\right)
\end{align}
in the $L^2$-sense. The assertion of Theorem \ref{Einsteinspectrumsincone} now follows from Lemma \ref{V_kEinsteinoperator}, Lemma \ref{Wmu_Einsteinoperator} and Lemma \ref{X_lambdaEinsteinoperator} below.
\begin{lem}\label{V_kEinsteinoperator}
	The spectrum of $\wt{\Delta}_E$ restricted to $V_{\kappa}$ is given by 
	\begin{align}
\left\{\eta_{n+1}(\xi_n(\kappa)+j)\mid j\in \N_0\right\},
	\end{align}
	and each of these eigenvalues has the same multiplicity as $\kappa$ as an eigenvalue of $\Delta_E$.	
\end{lem}
\begin{proof}
Let $h\in TT(M)$ and $\ol{h}\in C^{\infty}(S^2\overline{M})$ be of the form $\ol{h}=\varphi\cdot r^2h$ for some $\varphi\in C^{\infty}((0,r))$. By Lemma \ref{ttlemma}, this tensor satisfies
\begin{align}
\ol{\trace}\ol{h}=0,\qquad \ol{\delta}(\ol{h})=0,\qquad \ol{h}(\partial_r,.)=0,
\end{align}
and
\begin{align}
\overline{\Delta}_E\ol{h}=r^2(-\partial^2_{rr}\varphi\cdot h-n\cdot r^{-1}\partial_r\varphi\cdot h+\varphi\cdot r^{-2}\Delta_Eh),
\end{align}
so that, if $\Delta_Eh=\kappa\cdot h$, 
we have $\overline{\Delta}_E\ol{h}=0$, if $\varphi=r^{m}$, where $m=\xi_n(\kappa)$. Now consider some spaces of tensors on $\wh{M}$, given by
\begin{align*}
P_{m,2j}:=\left\{\sum_{l=0}^ja_l r^{m+2l}z^{2(j-l)}\ol{h}\mid a_l\in\R \right\},\qquad P_{m,2j+1}:=\left\{\sum_{l=0}^ja_l r^{m+2l}z^{2(j-l)+1}\ol{h} \mid a_l\in\R\right\}.
\end{align*}
It is straightforward to show that $P_{m,0}\oplus P_{m,1}\subset\ker(\wh{\Delta}_E)$ and that
for each $j\in \N$, $j\geq 2$, the Laplacian satisfies $\wh{\Delta}_E(P_{m,j})\subset P_{m,j-2}$. Because
$\dimn P_{m,j}=\dimn P_{m,j-2}+1$, $H_{m,j}:=P_{m,j}\cap \kernel\wh{\Delta}_E\neq\left\{0\right\}$ and it is also not hard to see that $H_{m,j}$ is one-dimensional for each $j\in \N_0$. Thus we obtain a sequence of $\wh{h}_j\in P_{m,j}$, hence $\wh{h}_j\in C^{\infty}(S^2\wh{M})\cap \ker(\wh{\Delta}_E)\cap TT(\wh{M})$, $j\in\N_0$, which can be written with respect to the variable $s=\sqrt{r^2+z^2}$ as
\begin{align}
\wh{h}_j=s^{m+j}\cdot s^2\wt{h}_j,\qquad  \wt{h}_j\in V_{\kappa},
\end{align}
and because for $\varphi=\varphi(s)\in C^{\infty}((0,\infty))$, Lemma \ref{ttlemma} implies
\begin{align}
\wh{\Delta}_E(\varphi\cdot s^2\wt{h}_j)=s^2(-\partial^2_{ss}\varphi\cdot\wt{h}_j-(n+1)\cdot s^{-1}\partial_s\varphi\cdot\wt{h}_j+\varphi\cdot s^{-2}\widetilde{\Delta}_E\wt{h}_j),
\end{align}
we get $\wt{\Delta}_E\wt{h}_j=\eta_{n+1}(m+j)\cdot\wt{h}_j$.
A density argument similar as in the proof of Theorem \ref{sinconefunctions} shows that the $\wt{h}_j$ span all of $V_{\kappa}$.
\end{proof}
\begin{lem}\label{Wmu_Einsteinoperator}
	The spectrum of $\wt{\Delta}_E$ restricted to $W_{\mu}$ is given by 
	\begin{align}
	\left\{\eta_{n+1}(\xi_n(\mu+1)+j)\mid j\in \N_0\right\},
	\end{align}
	and each of these eigenvalues has the same multiplicity as $\mu$ as an eigenvalue of $\Delta_1$.	
\end{lem}
\begin{proof}
If $\wt{h}\in W_{\mu}$ satisfies $\wt{\Delta}_E\wt{h}=\nu\cdot \wt{h}$ for some $\nu\geq -\frac{n^2}{4}$, Lemma \ref{ttlemma} implies that the tensor $\wh{h}=s^2\wt{h}$ satisfies
\begin{align}
\wh{\trace}(s^k\wh{h})=0,\qquad \wh{\delta}(s^k\wh{h})=0,\qquad s^k\wh{h}(\partial_s,.)=0,\qquad\wh{\Delta}_E(s^k\wh{h})=0,
\end{align} 
if $k=\xi_{n+1}(\nu)$.
To find such a tensor $\wh{h}$, we make the ansatz
\begin{align}\label{definitionPhi}
\Phi:=s^k\wh{h}=P[r,z]dz\odot r\omega+Q[r,z]dr\odot r\omega+R[r,z]r^2\delta^*\omega.
\end{align}
Demanding
\begin{align}
s^k\wh{h}\in TT(\wh{M}),\qquad s^k\wh{h}(\partial_s,.)=0,
\end{align}
 implies by direct computation (see also Lemma \ref{coupled1forms} in the appendix)
\begin{align}\label{determineQ,R}
Q[r,z]=-\frac{z}{r}P[r,z],\qquad r\cdot\partial_zP+r\cdot \partial_rQ+(n+1)Q=\frac{1}{2}(\mu-(n-1))R.
\end{align}
Thus, a choice of $P$ uniquely determines $Q$ and $R$.
We now want to translate the equation $\wh{\Delta}_E(s^k\wh{h})=0$ into an equation for $P,Q$ and $R$.
By using \eqref{divfree1formsA2}, we obtain
\begin{align}
\wh{\Delta}_E(P\cdot dz\odot r\omega)=\wh{\Delta}_2(P\cdot dz\odot r\omega)=dz\odot\wh{\Delta}_1(P\cdot r\omega)=(\wh{\Delta}_0P+r^{-2}(\mu+1)P)dz\odot r\omega.
\end{align}
We have
\begin{align}
\wh{\Delta}_E(dr\odot r\omega)=\ol{\Delta}_0(dr\odot r\omega)=-4\delta^*\omega+(\mu+n+3) dr\odot r^{-1}\omega,
\end{align}
where the second equation is \eqref{dm1}.
Moreover $dr\odot r\omega$ is parallel in the direction of $r$ and $z$, so that
\begin{equation}
\begin{split}
\wh{\Delta}_E(Qdr\odot r\omega)&=\wh{\Delta}_0Q\cdot dr\odot r\omega +Q[-4\delta^*\omega+(\mu+n+3) dr\odot r^{-1}\omega]\\
&=[\wh{\Delta}_0Q+r^{-2}(\mu+n+3)Q] dr\odot r\omega-4Q\cdot\delta^*\omega.
\end{split}
\end{equation}
Because $(\ol{M},\ol{g})$ is Ricci flat, \eqref{commutation} implies $\ol{\Delta}_E\circ \ol{\delta}^*=\ol{\delta}^*\circ\ol{\Delta}_1$. Using this relation,\eqref{divfree1formsA2} and \eqref{dm2}, we obtain
\begin{equation}\begin{split}
\wh{\Delta}_E(r^2\delta^*\omega)=\ol{\Delta}_E(r^2\delta^*\omega)=\ol{\Delta}_E\ol{\delta}^*(r^2\omega)=\ol{\delta}^*\ol{\Delta}_1(r^2\omega)&=(\mu+1-n)\ol{\delta}^*\omega\\
&=(\mu+1-n)[{\delta}^*\omega-r^{-2}dr\odot r\omega].
\end{split}
\end{equation}
The tensor $r^2\delta^*\omega$ is also parallel in the direction of $r$ and $z$. Thus,
\begin{align}
\wh{\Delta}_E(R\cdot r^2\delta^*\omega)=\wh{\Delta}_0R\cdot  r^2\delta^*\omega+(\mu+1-n)R[{\delta}^*\omega-r^{-2}dr\odot r\omega].
\end{align}
 Therefore, demanding that $\wh{\Delta}_E\Phi=0$ (where $\Phi$ was defined in \eqref{definitionPhi}) yields the system of equations
\begin{align}
\label{eq1.1}\wh{\Delta}P+r^{-2}(\mu+1)P&=0,\\
\label{eq1.2}\wh{\Delta}Q+r^{-2}(\mu+n+3)Q-r^{-2}(\mu+1-n)R&=0,\\
\label{eq1.3}\wh{\Delta}R+r^{-2}(\mu+1-n)R-4r^{-2}Q&=0.
\end{align}
In Lemma \ref{formulas2} below, we show that if $P$ solves \eqref{eq1.1} and $Q,R$ are defined by \eqref{determineQ,R}, then $Q,R$ solve the system \eqref{eq1.2}-\eqref{eq1.3}. In fact, \eqref{eq1.1} exactly says that $P\cdot r\omega$ is a harmonic $1$-form. We have considered harmonic $1$-forms in the previous section. In particular we found out that for each $j\in \N$, there is a solution $P_j$ of \eqref{eq1.1} which scales as
\begin{align}
P_j[\alpha r,\alpha z]=\alpha^{l+j}P_j[r,z],
\end{align}
where $l=\xi_{n+1}(\mu+1)$. By definition, $Q_j$, and $R_j$ obtained from $P_j$ by \eqref{determineQ,R} admit the same scaling behaviour. Therefore, with respect to the coordinate $s=\sqrt{r^2+z^2}$ one obtains a tensor $\wh{h}_j\in TT(\wh{M})$ in the kernel of $\wh{\Delta}_E$, which can be written as $s^{l+j}s^2\wt{h}_j$ for some $\wt{h}_j\in TT(\widetilde{M})$,
By Lemma \ref{ttlemma},
\begin{align}
\wh{\Delta}_E(\varphi\cdot s^2\wt{h})=s^2(-\partial^2_{ss}\varphi\cdot \wt{h}-(n+1)\cdot s^{-1}\partial_s\varphi\cdot \wt{h}+\varphi s^{-2}\wt{\Delta}_E\wt{h}),
\end{align}
for $\varphi=\varphi(s)\in C^{\infty}((0,\infty))$ and $\wt{h}\in TT(\widetilde{M})$, so that $\wt{\Delta}_E\wt{h}_j=\eta_{n+1}(l+j)\cdot\wt{h}_j$.
Since $\wt{h}_j\in W_\mu$ by construction, we have obtained eigenvalues of $\wt{\Delta}_E$ restricted to this space. A refined density argument as in the proof of Lemma \ref{Wmuoneforms} implies that these are all the eigenvalues.
\end{proof}
\begin{lem}\label{X_lambdaEinsteinoperator}
	The spectrum of $\wt{\Delta}_E$ restricted to $X_{\lambda}$ is given by 
	\begin{align}
	\left\{\eta_{n+1}(\xi_n(\lambda)+j)\mid j\in \N_0\right\},
	\end{align}
	and each of these eigenvalues has the same multiplicity as $\lambda$ as an eigenvalue of $\Delta_0$.	
\end{lem}
\begin{proof}
	If $\wt{h}\in X_{\lambda}$ satisfies $\wt{\Delta}_E\wt{h}=\nu\cdot \wt{h}$ for some $\nu\geq -\frac{n^2}{4}$, Lemma \ref{ttlemma} implies that the tensor $\wh{h}=s^2\wt{h}$ satisfies
	\begin{align}
	\wh{\trace}(s^k\wh{h})=0,\qquad \wh{\delta}(s^k\wh{h})=0,\qquad s^k\wh{h}(\partial_s,.)=0,\qquad\wh{\Delta}_E(s^k\wh{h})=0,
	\end{align} 
	if $k=\xi_{n+1}(\nu)$.
Because $\wt{h}\in X_{\lambda}$, $s^k\wh{h}$ is necessarily of the form
	\begin{equation}\begin{split}
	s^k\wh{h} &=P^{(1)}[r,z]vdz\otimes dz+P^{(2)}vdz\odot dr+P^{(3)}[r,z]vdr\otimes dr+Q^{(1)}[r,z]dz\odot rdv\\
	&\qquad+Q^{(2)}[r,z]dr\odot rdv+R[r,z]r^2(n\nabla^2v+\Delta v g)+S[r,z]vr^2g.
	\end{split}
	\end{equation}
		Demanding $s^k\wh{h}(\partial_s,.)=0$  yields the relations
	\begin{align}\label{determineP's}
	P^{(2)}=-\frac{z}{r}P^{(1)},\qquad P^{(3)}=-\frac{z}{r}P^{(2)}=\frac{z^2}{r^2}P^{(1)},\qquad Q^{(2)}=-\frac{z}{r}Q^{(1)},
	\end{align}
	and demanding  $\ol{\trace}s^k\wh{h}=0$ yields
	\begin{align}
\label{determineS}P^{(1)}+	P^{(3)}+nS=0.
\end{align}	
Lemma \ref{philemma} below shows that demanding $\ol{\delta}\Phi=0$ is equivalent to the system of equations
	\begin{align}
\label{determineQ_1}r\cdot\partial_zP^{(1)}+r\cdot\partial_rP^{(2)}&=\lambda Q^{(1)}-nP^{(2)},\\
\label{determineQ_2}r\cdot\partial_zP^{(2)}+r\cdot\partial_rP^{(3)}&=\lambda Q^{(2)}-nP^{(3)}+n S,\\
\label{determineR}r\cdot\partial_zQ^{(1)}+r\cdot\partial_rQ^{(2)}&=(n-1)(\lambda-n) R-(n+1)Q^{(2)}-S.
	\end{align}
Given $P^{(1)}$, one defines $P^{(2)},P^{(3)}$ and $S$ by \eqref{determineP's},\eqref{determineS} and $Q^{(1)},Q^{(2)}$ and $R$ by \eqref{determineQ_1},\eqref{determineQ_2} and \eqref{determineR}. It is not hard to see that the definitions of $Q^{(1)}$ and $Q^{(2)}$ are consistent with the last equation in \eqref{determineP's}.
Let us now translate the equation $\wh{\Delta}_E(s^k\wh{h})=0$ into an equation for $P^{(i)},Q^{(i)},R$ and $S$.
 We compute,
using \eqref{delta1P}, \eqref{delta1Q} and \eqref{delta1R},
\begin{equation}
\begin{split}
\wh{\Delta}_E(P^{(1)}vdz\otimes dz)&=\wh{\Delta}_0(P^{(1)}v)dz\otimes dz=(\wh{\Delta}_0P^{(1)}+\lambda r^{-2}P^{(1)})vdz\otimes dz,\\
\wh{\Delta}_E(P^{(2)}vdz\odot  dr)&=(\wh{\Delta}_1P^{(2)}vdr)\odot dz\\
&=(\wh{\Delta}_0P^{(2)}+(\lambda+n)r^{-2}P^{(2)})vdz\odot dr-2r^{-2}P^{(2)}dz\odot rdv,\\
\wh{\Delta}_E(Q^{(1)} dz\odot rdv)&=(\wh{\Delta}_1Q^{(1)}  rdv)\odot dz\\
&=(\wh{\Delta}_0Q^{(1)}+(\lambda-n+2)r^{-2}Q^{(1)})dz\odot rdv-2r^{-2}Q^{(1)}\lambda vdz\odot dr.
\end{split}
\end{equation}
Therefore, $\wh{\Delta}_E\Phi=0$ implies
\begin{align}
\label{eqP_1}\wh{\Delta}P^{(1)}+\lambda r^{-2}P^{(1)}&=0,\\
\label{eqP_2}\wh{\Delta}P^{(2)}+(\lambda+n)r^{-2}P^{(2)}-2r^{-2}\lambda Q^{(1)}&=0,\\
\label{eqQ_1}\wh{\Delta}Q^{(1)}+(\lambda-n+2)r^{-2}Q^{(1)}-2r^{-2}P^{(2)}&=0,
\end{align}
which actually coincides with the system \eqref{eq1.1}-\eqref{eq1.3}.
Furthermore, by \eqref{gradhessdr1} and \eqref{lem1},
\begin{align}
\ol{\nabla}dr=r\cdot g,\qquad \ol{\Delta}dr=nr^{-2}dr,\qquad \ol{\Delta}(dr\otimes dr)=2r^{-2}(ndr\otimes dr-r^2g),
\end{align}
so that
\begin{equation}\begin{split}\label{deltaE1}
\wh{\Delta}_E(P^{(3)}vdr\otimes dr)&=(\wh{\Delta}_0P^{(3)})vdr\otimes dr+P^{(3)}\ol{\Delta}_2(vdr\otimes dr)-2
 \langle\wh{\nabla}P^{(3)},\wh{\nabla}(vdr\otimes dr)\rangle_{\wh{g}}\\
&=(\wh{\Delta}_0P^{(3)}+r^{-2}(\lambda+2n)P^{(3)})vdr\otimes dr\\&\qquad-2P^{(3)}vg-2r^{-2}P^{(3)}dr\odot rdv.
\end{split}
\end{equation}
By \eqref{lem2},
\begin{align}
\ol{\Delta}_E(r^2g)=\ol{\Delta}_2(r^2g)=2r^{-2}(r^2g-ndr\otimes dr),
\end{align}
which implies
\begin{equation}
\begin{split}
\label{formulaS}
\wh{\Delta}_E(Svr^2g)&=\wh{\Delta}_0(Sv)r^2g+Sv\ol{\Delta}_E(r^2g)-2\langle \wh{\nabla}(Sv), \wh{\nabla}(r^2g)\rangle_{\wh{g}}\\
&=(\wh{\Delta}_0S+r^{-2}(\lambda+2)Sv)r^2g-2nr^{-2}Svdr\otimes dr+2Sr^{-2}dr\odot rdv.
\end{split}
\end{equation}
Moreover, by \eqref{lem3}
\begin{align}
\ol{\Delta}(dr\odot rdv)&=-4r^{-2}\Delta_0 v\cdot dr\otimes dr+r^{-2}dr\odot rd(\Delta_0 v+4v)-4\nabla^2v,
\end{align}
from which we obtain, since $dr\odot rdv$ is parallel in the directions of $r$ and $z$ that
\begin{equation}\begin{split}
\wh{\Delta}_E(Q^{(2)}dr\odot rdv)&=[\wh{\Delta}_0Q^{(2)}+(\lambda+4)r^{-2}Q^{(2)}]dr\odot rdv\\
&\qquad-4r^{-2}Q^{(2)}\lambda vdr\otimes dr-4Q^{(2)}\nabla^2v.
\end{split}
\end{equation}
It remains to compute the equation for the $R$ part. For this purpose, we note that by \eqref{lem4}-\eqref{lem6},
\begin{equation}\begin{split}
\ol{\nabla}^2(r^2v)&=r^2\nabla^2v+2v\ol{g}+dr\otimes rdv,\qquad \ol{\nabla}^2v=\nabla^2v-r^{-2}dr\odot rdv,\\ \ol{\Delta}_0(r^2v)&=(\lambda-2n-2)v,
\end{split}
\end{equation}
and because $(\ol{M},\ol{g})$ is Ricci flat, the commutation relation $\ol{\Delta}_E\circ \ol{\nabla}^2=\ol{\nabla}^2\circ\ol{\Delta}_0$ follows from \eqref{commutation}. Using these formulas and \eqref{lem3} again, we have
\begin{equation}\begin{split}
\ol{\Delta}_E(r^2\nabla^2v)&=\ol{\Delta}_E(\ol{\nabla}^2(r^2v)-2v\ol{g}-dr\odot rdv)\\
&=\ol{\nabla}^2\ol{\Delta}_0(r^2v)-2r^{-2}\lambda v\ol{g}-[(\lambda+4)r^{-2}dr\odot rdv-4\nabla^2v-4\lambda r^{-2}vdr\otimes dr]\\
&=(\lambda-2n+2)\ol{\nabla}^2v-2r^{-2}\lambda v(dr\otimes dr+r^2g)\\&\qquad -[(\lambda+4)r^{-2}dr\odot rdv-4\nabla^2v-4\lambda r^{-2}vdr\otimes dr]\\
&=(\lambda-2n+2)(\nabla^2v-r^{-2}dr\odot rdv)-2r^{-2}\lambda v(dr\otimes dr+r^2g)\\&\qquad -[(\lambda+4)r^{-2}dr\odot rdv-4\nabla^2v-4\lambda r^{-2}vdr\otimes dr]\\
&=(\lambda-2n+2)\nabla^2v-2(\lambda-n+1)r^{-2}dr\odot rdv-2\lambda vg+2\lambda r^{-2}vdr\otimes dr.
\end{split}
\end{equation}
From \eqref{formulaS}, we obtain
\begin{align}
\ol{\Delta}_E(r^2\Delta vg)=-2n\lambda r^{-2}vdr\otimes dr+\lambda(\lambda+2)vg+2\lambda r^{-2}dr\odot rdv,
\end{align}
and by adding up, we get
\begin{equation}\begin{split}
\wh{\Delta}_E(Rr^2(n\nabla^2v+\Delta vg))&=(\wh{\Delta}_0R+r^{-2}(\lambda-2n+2)R)r^2(n\nabla^2v+\Delta_0 v\cdot g)\\
&\qquad+2(n-1)(n-\lambda)Rr^{-2}dr\odot rdv.
\end{split}
\end{equation}
Therefore, $\wh{\Delta}_E\Phi=0$ is equivalent to the system \eqref{eqP_1}-\eqref{eqQ_1} together with the equations
\begin{align}
\label{eqP_3}\wh{\Delta}P^{(3)}+(\lambda+2n)r^{-2}P^{(3)}-2nSr^{-2}-4\lambda r^{-2}Q^{(2)}&=0,\\
\label{eqS}\wh{\Delta}S+(\lambda+2)r^{-2}S-2r^{-2}P^{(3)}+\frac{4}{n}r^{-2}\lambda Q^{(2)}&=0,\\
\label{eqQ_2}\wh{\Delta}Q^{(2)}+(\lambda+4)r^{-2}Q^{(2)}-2r^{-2}P^{(3)}+2r^{-2}S+2(n-1)(n-\lambda)r^{-2}R&=0,\\
\label{eqR}\wh{\Delta}R+(\lambda-2n+2)r^{-2}R-\frac{4}{n}r^{-2}Q^{(2)}&=0.
\end{align}
Lemma \ref{formulas3} in the appendix shows that if $P^{(1)}$ satisfies \eqref{eqP_1} and $P^{(2)},P^{(3)},Q^{(1)},Q^{(2)},R,S$ are defined by \eqref{determineP's}-\eqref{determineR}, then the equations \eqref{eqP_2},\eqref{eqQ_1} and \eqref{eqP_3}-\eqref{eqR} also hold.
In fact, \eqref{eqP_1} exactly says that $P^{(1)}\cdot v$ is a harmonic function. In the proof of Theorem \ref{sinconefunctions}, we found that for each $j\in \N$  a solution $P^{(1)}_j$ of \eqref{eqP_1} which scales as
\begin{align}
P^{(1)}_j[\alpha r,\alpha z]=\alpha^{k+j}P^{(1)}_j[r,z],
\end{align}
where $k=\xi_{n}(\lambda)$. By definition, $P^{(2)}_j,P^{(3)}_j,Q^{(1)}_j,Q^{(2)}_j,R_j,S_j$  defined by \eqref{determineP's}-\eqref{determineR} admit the same scaling behaviour. Therefore, one obtains a tensor $\wh{h}_j\in TT(\wh{M})$ in the kernel of $\wh{\Delta}_E$, which can be written with respect to the coordinate $s=\sqrt{r^2+z^2}$ as $\wh{h}_j=s^{k+j}s^2\wt{h}_j$ for some $\wt{h}_j\in TT(\widetilde{M})$.
Due to Lemma \ref{ttlemma}, 
\begin{align}
\wh{\Delta}_E(\varphi\cdot s^2\wt{h})=s^2(-\partial^2_{ss}\varphi\cdot \wt{h}-(n+1)\cdot s^{-1}\partial_s\varphi\cdot \wt{h}+\varphi s^{-2}\wt{\Delta}_E\wt{h}),
\end{align}
for $\varphi=\varphi(s)\in C^{\infty}((0,\infty))$ and $\wt{h}\in TT(\widetilde{M})$. Thus, we get $\wt{\Delta}_E\wt{h}_j=\eta_{n+1}(k+j)\cdot\wt{h}_j$.
Since $\wt{h}_j\in X_\lambda$ by construction, we have obtained eigenvalues of $\wt{\Delta}_E$ restricted to this space.  A refined density argument as in the proof of Lemma \ref{Wmuoneforms} implies that these are all the eigenvalues.
\end{proof}
\begin{proof}[Proof of Theorem \ref{mainthm:linstability}]
This is a direct application of the assertions in Theorem \ref{Einsteinspectrumsincone} and we use the notation thereof. For (i), note that $\wt{\lambda}^{(3)}_{i,j}$ and $\wt{\mu}^{(2)}_{i,j}$ are positive for $(i,j)\in \N\times\N_0$ because $\lambda_i$ and $\mu_i$ are positive for $i\in\N$. 
Therefore, (strict) EH-stability of $(\wt{M},\wt{g})$ is equivalent to 
$\wt{\kappa}^{(1)}_{i,j}\geq0$ 
(resp.\ $\wt{\kappa}^{(1)}_{i,j}>0$ )
 for all $(i,j)\in \N\times\N_0$, which in turn is equivalent to $\kappa_i\geq0$ (resp.\ $\kappa_i>0$) for all $i\in\N$. 
 For (ii), we use the characterization of Remark \ref{rem:stability} (i).
By Theorem \ref{sinconefunctions}, $\spectrum_+(\wt{\Delta)}\geq 2n$ (resp.\ $\spectrum_+(\wt{\Delta)}> 2n$) is equivalent to $\spectrum_+({\Delta})\geq 2n-\frac{n}{2}(\sqrt{1+\frac{8}{n}}-1)$ (resp. $\spectrum_+({\Delta})> 2n-\frac{n}{2}(\sqrt{1+\frac{8}{n}}-1)$). Together with (i), this implies (ii). Similarly for (iii), we use characterization of Remark \ref{rem:stability} (i).
By Theorem \ref{sinconefunctions}, $\spectrum_+(\wt{\Delta)}\geq 2(n+2)$ (resp.\ $\spectrum_+(\wt{\Delta})> 2(n+2)$) is equivalent to $\spectrum_+({\Delta)}\geq 2(n+1)$ (resp. $\spectrum_+({\Delta})> 2(n+1)$). Together with (i), this implies (iii). Finally, (iv) follows from Lemma \ref{unbounded}, Theorem \ref{Einsteinspectrumsincone} and the fact that $\kappa_i\geq-\frac{(n-1)^2}{4}$ for all $i\in\N$ is equivalent to $\wt{\kappa}^{(1)}_{i,j}\geq -\frac{n^2-1}{4}>-\frac{n^2}{4}$ for all $(i,j)\in \N\times \N_0$.
\end{proof}
\begin{proof}[Proof of Corollary \ref{maincor:stability}]
	If $(M,g)=(M_1^{n_1},g_1)\times(M_2^{n_2},g_2)$, $n_1,n_2\geq2$, an easy calculation shows $h=n_2g_1-n_1g_2\in TT(M)$ and $\Delta_Eh=-2(n-1)h$. If $n\leq 8$, $-2(n-1)<-\frac{1}{4}(n-1)^2$ and $(M,g)$ is physically unstable. Therefore, by Theorem \ref{mainthm:linstability} (iv), $\wt{\Delta}_E$ is unbounded below. If $(M_1^{n_1},g_1)$ and $(M_2^{n_2},g_2)$ are both linearly stable, $-2(n-1)$ is the only negative eigenvalue, see \cite[Proposition 4.8]{Kro15b}. Therefore, if $n\geq9$, $(M^n,g)$ is physically stable and by Theorem \ref{mainthm:linstability} (iv), $\wt{\Delta}_E$ is bounded below.
\end{proof}
\section{Stability under the singular Ricci-de Turck flow}
 Well-posedness of Ricci-de Turck flow in the setting of conifolds was established in \cite{Ver16}. It turned out that in order to establish shorttime existence and uniqueness, the conical singularities had to be modelled over Ricci-flat cones with tangentially stable cross-sections (strictly tangentially stable in the non-orbifold case). For convenience, we call such singularities (strictly) tangentially stable. The topologies on the space of metrics in this subsection are induced by hybrid weighted H\"{o}lder norms $\mathcal{H}^{k,\alpha}_{\gamma}$ for $k\geq2$,$\gamma>0$ and $\alpha\in (0,1)$, see \cite[Definition 4.5]{KV18}.

In \cite{KV18}, Boris Vertman and the author established a stability theorem for compact Ricci-flat conifolds under the singular Ricci-de Turck flow. The methods used in this paper can also be directly adapted to the stability of Einstein manifolds with isolated conical singularities under the volume normalized Ricci flow. Then we arrive at the following result:
\begin{thm}\label{dynstability}Let $({M}^n,{g})$, $n\geq 3$ be compact Einstein conifold, where the all the singularities are either orbifold singularities or strictly tangentially stable. Then, if  $({M},{g})$ is strictly linearly stable,
 there exists a small neighbourhood $\mathcal{U}$ around $g$ such that any normalized singular Ricci de Turck flow $g(t)$ starting in $\mathcal{U}$ exists for all time and converges to $c\cdot g$ with $c$ close to $1$.
	\end{thm}
\begin{proof}[Sketch of proof]
The evolution equation for the normalized Ricci-de Turck flow is
\begin{align}
\partial_t g(t)=-2\ric_{g(t)}+\mathcal{L}_{V(g(t),g)}g(t)+\frac{2}{n\cdot \volume(g(t))}\int_M\scal_{g(t)}\dv_{g(t)},
\end{align}
where the de Turck vector field is defined by $V(g(t),g)^k=g^{ij}(\Gamma(g)_{ij}^k-\Gamma(g)_{ij}^k)$. Note that the total scalar curvature is defined because $\scal_g=O(x^{-2})$ in a neighbourhood of isolated conical singularities and thus integrable. Well-posedness for the singular Ricci-de Turck flow in a small neighbourhood with respect to an appropriate topology is granted by \cite{Ver16} and the condition of tangential stability. Well-posedness for the volume normalized version follows by projecting a solution of the unnormalized flow to the space of metrics of fixed volume and a suitably rescaling in time as well. For $h(t)=g(t)-g$ and if $g$ is Einstein, the unnormalized Ricci-de Turck flow is written as
\begin{equation}\begin{split}
\partial_th&=-\Delta_Eh+(g+h)^{-1}*(g+h)^{-1}*\nabla h*\nabla h+(g+h)^{-1}*(g+h)^{-1}*\nabla^2h*\nabla h\\
&=:-\Delta_Eh+Q_2(h).
\end{split}
\end{equation}
The normalized version can be written as
\begin{align}\label{normricciflow}
\partial_th=-\Delta_Eh+Q_2(h)+\frac{1}{n}\fint \trace(\Delta_Eh-Q_2(h))\dv,
\end{align}
where $\fint$ denotes the mean value integral with respect to $g+h$. Now it amounts to decompose the evolving metric $g(t)=h(t)+g$ uniquely as $g(t)=(c(t)+1)g+k(t)$ where $c(t)\in\R$ and  $\fint \trace_{g}k(t)\dv_{g}=0$, i.e.\ $k\in g^{\perp}$.
The evolution equation \eqref{normricciflow} can be rewritten in terms of $c$ and $k$ and the condition ${\Delta}_E|_{{g}^{\perp}}>0$ ensures good properties of the heat kernel of the associated heat equation. The technical details to prove convergence are carried out as in \cite[Section 11]{KV18}.
	\end{proof}
\begin{rem}
	In \cite[Section 11]{KV18}, we change the reference metric of the Ricci de Turck flow at discrete times so that the resulting family of metrics is only piecewise smooth. While carrying out the details in this setting, we also have to change the reference metric, but the resulting de Turck vector field does not change because all reference metrics are of the form $c\cdot g$, $c\in\R$. Therefore, the flow will be smooth in this case.
\end{rem}
\begin{proof}[Proof of Theorem \ref{mainthm:stability}]
For a smooth Einstein manifold $(M,g)$, its sine-cone $(\wt{M},\wt{g})$ is a conifold with two isolated conical singularities modelled over the Ricci flat cone $(\ol{M},\ol{g})$. By assumption, the singularities are both strictly tangentially stable. Remark \ref{rem:stability} shows that strict tangential stability implies strict linear stability.
By Theorem \ref{mainthm:linstability}, $(\wt{M},\wt{g})$
is strictly tangentially stable and by Remark \ref{rem:stability}, it is also strictly linearly stable. The assertion follows from Theorem \ref{dynstability}.
\end{proof}

%

\section{Infinitesimal deformability of sine-cones}
\begin{proof}[Proof of Theorem \ref{mainthm:rigidity}]
Note that part (ii) is a direct consequence of Theorem \ref{mainthm:linstability} (i). For part (i), we use the notation of Theorem \ref{Einsteinspectrumsincone}.
Recall again that $\wt{\lambda}^{(3)}_{i,j}$ and $\wt{\mu}^{(2)}_{i,j}$ are positive for $(i,j)\in \N\times\N_0$ because $\lambda_i$ and $\mu_i$ are positive for $i\in\N$. 
Thus in order to get all infinitesimal Einstein deformations on $(\wt{M},\wt{g})$, we have to find all zeros in the set
\begin{align*}
\left\{\wt{\kappa}^{(1)}_{i,j}\mid(i,j)\in \N\times\N_0\right\}.
\end{align*}
If $(M,g)$ has infinitesimal Einstein deformations, then $\kappa_{i_0}=0$ for some $i_0\in\N$. In this case, $\kappa^{(1)}_{i_0,0}=0$ by definition.  
 Let $h\in E(\Delta_E,\kappa_{i_0})$ and let us go through the construction of Lemma \ref{V_kEinsteinoperator}. Then $\ol{h}=r^2h$ is a harmonic tensor on $(\ol{M},\ol{g})$, which extends to a harmonic tensor $\wh{h}=\ol{h}=r^2h$ on $(\wh{M},\wh{g})$. 
By writing $\wh{h}=\sin(\theta)^2s^2h$, we see that it restricts to the bounded infinitesimal Einstein deformation $\wt{h}=\sin(\theta)^2h$ on $(\wt{M},\wt{g})$. 
Now if $(M,g)$ does not admit infinitesimal Einstein deformations, we can only have $\kappa^{(1)}_{i_0,j_0}=0$, for some $(i_0,j_0)\in \N\times\N_0$, if $\kappa_{i_0}<0$.
 In this case, the corresponding IED is obtained as follows: If $h\in E(\Delta_E,\kappa_{i_0})$, then $r^m(r^2h)=r^m\ol{h}$ is harmonic on $(\ol{M},\ol{g})$ with $m=\xi(\kappa_{i_0})<0$. The infinitesimal Einstein deformation $\wt{h}\in E(\wt{\Delta}_E,\kappa^{(1)}_{i_0,j_0})$, comes from restricting a harmonic tensor $\wh{h}$ on $(\wh{M},\wh{g})$ of the form
 \begin{align*}
 \wh{h}=\sum_{l=0}^{j_0}a_lr^{m+2l}z^{2(j-l)+b}\ol{h},\qquad a_l\in \R
 \end{align*}
 to $(\wt{M},\wt{g})$. Here $b=0$ if $j_0$ is even and $b=1$, if $j_0$ is odd. We get
 \begin{align*}
\wt{h}= \sum_{l=0}^{j_0}a_l\sin(\theta)^{m+2l}\cos(\theta)^{2(j-l)+b}\sin(\theta)^2h.
 \end{align*}
 A careful comparison of coefficients shows that harmonicity 
 of $\wh{h}$ implies that $a_0\neq0$. Therefore, because $m<0$, $\wt{h}$ is unbounded on $(\wt{M},\wt{g})$.
 \end{proof}
\begin{proof}[Proof of Corollary \ref{maincor:rigidity}]
Since $(M,g)$ is a product manifold, $-2(n-1)\in\spectrum(\Delta_E|_{TT})$ by \cite[Proposition 4.8]{Kro15b}. If $n=9$, we get an eigenvalue $\kappa_{i_0}=-16$ and it is straightforward to see that $\kappa^{(1)}_{i_0,4}=0$ (where we use the notation of Theorem \ref{Einsteinspectrumsincone} again). Therefore $\wt{\Delta}_E|_{TT}$ admits a nontrivial kernel.
If $(M,g)$ is the product of strictly linearly stable Einstein manifolds, $\Delta_E\geq -2(n-1)$ and $-2(n-1)\in\spectrum(\Delta_E|_{TT})$ appears with multiplicity one \cite[Proposition 4.8]{Kro15b}. Moreover, this is the only nonpositive eigenvalue on $TT$-tensors.
Because $\wt{\lambda}^{(3)}_{i,j}>0$ and $\wt{\mu}^{(2)}_{i,j}>0$ for $i\in\N$ and $j\in\N_0$, it remains to check which $\kappa^{(1)}_{i,j}$ can be zero. At first, $\kappa_1=-2(n-1)$ and $\kappa_i>0$ for all $i\in\N$, $i\geq2$ so that $\kappa^{(1)}_{i,j}>0$ for all $i\geq2$ and $j\in \N$. Thus it remains to find all zeros in the sequence
\begin{align}
\kappa^{(1)}_{1,j}=(m_1+j)(m_1+j+n),\qquad j\in\N_0,\qquad m_1=-\frac{n-1}{2}+\sqrt{\frac{(n-1)^2}{4}+\kappa_1}.
\end{align}
We have
\begin{equation}
 \begin{split}
(m_1+j)(m_1+j+n)&=m_1(m_1+n-1)+m_1(j+1)+j(m_1+j+n)\\
&=\kappa_1+j(2m_1+j+n)+m_1\\
&=-2(n-1)+j(2m_1+j+n)+m_1,
\end{split}
\end{equation}
so that $\kappa^{(1)}_{1,j}=0$  is equivalent to
\begin{align}\label{eqm_1}
(2j+1)m_1=2(n-1)-j(j+n).
\end{align}
If $n=9$, $m_1=-4$ and $j=4$ is the only nonnegative integer solving this equation. For higher dimensions we observe that
\begin{align}
2m_1=[\sqrt{(n-9)(n-1)}-(n-1)],
\end{align}
so that \eqref{eqm_1} can be rewritten as
\begin{align}
(2j+1)[\sqrt{(n-9)(n-1)}-(n-1)]=4(n-1)-2j(j+n).
\end{align}
This equation implies that $\ell:=\sqrt{(n-9)(n-1)}\in\N$, otherwise, it would be irrational. But $n=5+\sqrt{16+\ell^2}$ and by the list of phytagorean triples $\sqrt{16+\ell^2}$ is only an natural number for $\ell=0$ (which corresponds to the case $n=9$ that we discussed already) and for $\ell=3$ (which implies $n=10$). In the latter case, $m_1=-3$ and \eqref{eqm_1} implies $j=-2+\sqrt{4+15}\notin\N$. This finishes the proof.
\end{proof}
\appendix
\section{Some formulas for Laplace-type operators on warped products}
In the following, we denote the indices corresponding to coordinates on $M$ by $i,j,k,\ldots$. The indices $r,s,\theta$ refer to the corresponding coordinates in the construction of the manifolds $\widetilde{M},\overline{M}$ and $\wh{M}$.
Let us denote the indices corresponding to coordinates on $\widetilde{M}$ by $a,b,c,\ldots$ and the indices corresponding to coordinates on $\overline{M}$ by $\alpha,\beta,\gamma$.
The christoffel symbols on $\widetilde{M}$ and $\overline{M}$ are related to the ones on $M$ by
\begin{align}\label{christoffelA1}
\widetilde{\Gamma}_{ij}^k&=\Gamma_{ij}^k,\qquad \widetilde{\Gamma}_{ij}^{\theta}=-\cos(\theta)\sin(\theta)g_{ij},\qquad
\widetilde{\Gamma}_{i\theta}^j=\widetilde{\Gamma}_{\theta i}^j=\frac{\cos(\theta)}{\sin(\theta)}\delta_i^j,\\
\label{christoffelA2}
\overline{\Gamma}_{ij}^k&=\Gamma_{ij}^k,\qquad \overline{\Gamma}_{ij}^{r}=-r\cdot g_{ij},\qquad
\overline{\Gamma}_{ir}^j=\overline{\Gamma}_{ri}^j=\frac{1}{r}\delta_i^j,
\end{align}
while the other Christoffel symbols vanish.
The Christoffel symbols of $\wh{M}$ and $\widetilde{M}$ are related by
\begin{align}
\wh{\Gamma}_{ab}^c&=\widetilde{\Gamma}_{ab}^c,\qquad \wh{\Gamma}_{ab}^{s}=-s\cdot\wt{g}_{ab},\qquad
\wh{\Gamma}_{as}^b=\wh{\Gamma}_{sa}^b=\frac{1}{s}\delta_a^b,
\end{align}
while the other Christoffel symbols vanish. The relation between the Christoffel symbols of $\wh{M}$ and $\overline{M}$ is simply $\wh{\Gamma}_{\alpha\beta}^{\gamma}=\overline{\Gamma}_{\alpha\beta}^{\gamma}$ while all terms containing at least one $z$ vanish. Consequently, the Laplace Beltrami operators of the four metrics are related by
\begin{align}
\wt{\Delta}&=-\partial^2_{\theta\theta}-n\sin(\theta)^{-1}\cos(\theta)\partial_\theta+\sin(\theta)^{-2}\Delta,\\
\label{olDelta}	\overline{\Delta}&=-\partial^2_{rr}-n\cdot r^{-1}\partial_r+r^{-2}\Delta,\\
	\wh{\Delta}&=-\partial^2_{ss}-(n+1)\cdot s^{-1}\partial_s+s^{-2}\widetilde{\Delta}\\
&=-\partial^2_{zz}+\overline{\Delta}=-\partial^2_{zz}-\partial^2_{rr}-n\cdot r^{-1}\partial_r+r^{-2}\Delta.	
\end{align}

\begin{lem}\label{tildelemma}
	Let $\chi,\psi\in C^{\infty}(0,\pi)$, $\omega\in D(\wt{M})$ and $v\in C^{\infty}$. Furthermore, consider $d\theta\in C^{\infty}(T^*\wt{M})$. Then we have the following identities
	\begin{align}
	\label{tildelemma0}\wt{\delta}(\psi d\theta)&=-\partial_{\theta}\psi-n\frac{\cos}{\sin}\psi,\\
	\label{tildelemma1}	\wt{\delta}(\psi\sin \omega\odot d\theta)&=-[\sin \partial_{\theta}\psi+(n+1)\cos\psi]\omega,\\
	\label{tildelemma2}	\wt{\delta}(\psi\sin dv\odot d\theta)&=\sin^{-1}\psi \Delta v\cdot d\theta-[\sin \partial_{\theta}\psi+(n+1)\cos\psi]dv,\\
	\label{tildelemma3}	\wt{\delta}(\chi v(nd\theta\otimes d\theta-\sin^2g))&=\chi dv-n[\partial_{\theta}\chi+(n+1)\frac{\cos}{\sin}\chi]v d\theta.
	\end{align}
\end{lem}
\begin{proof}For \eqref{tildelemma0}, we compute
	\begin{align*}
	\wt{\nabla}_{\theta}(\psi d\theta)_{\theta}=\partial_{\theta}\psi,\qquad \wt{\nabla}_{i}(\psi d\theta)_{j}=\sin\cos\psi g_{ij},
	\end{align*}
	from which the result follows by taking the trace.	To prove \eqref{tildelemma1} and \eqref{tildelemma2}, we first let $\eta\in C^{\infty}(T^*{M})$ be arbitrary. Then, we compute
	\begin{equation}
	\begin{split}
	\wt{\nabla}_i(\psi\sin\eta\odot d\theta)_{jk}&=\psi\cos\sin^2(\eta_j\cdot g_{ik}+\eta_k\cdot g_{ij}),\qquad
	\wt{\nabla}_{\theta}(\psi\sin\eta\odot d\theta)_{\theta k}=\sin\partial_{\theta}\psi \cdot\eta,\\
	\wt{\nabla}_i(\psi\sin\eta\odot d\theta)_{j\theta}&=\sin\nabla_i\eta_j\cdot\psi,\qquad  	\wt{\nabla}_{\theta}(\psi\sin\eta\odot d\theta)_{\theta\theta}=0,
	\end{split}
	\end{equation}
	and \eqref{tildelemma1} and \eqref{tildelemma2} follow by taking the trace and either inserting $dv$ or $\omega$. For \eqref{tildelemma3}, we compute
	\begin{equation}
	\begin{split}
	\wt{\nabla}_{\theta}(\chi v d\theta\otimes d\theta)_{\theta\theta}&=\partial_{\theta}\chi\cdot v,\qquad 
	\wt{\nabla}_{i}(\chi v d\theta\otimes d\theta)_{j\theta}=\chi v\cos\sin g_{ij},\\
	\wt{\nabla}_{i}(\chi v d\theta\otimes d\theta)_{jk}&=\wt{\nabla}_{\theta}(\chi v d\theta\otimes d\theta)_{\theta k}=0,
	\end{split}
	\end{equation}
	and
	\begin{equation}
	\begin{split}
	\wt{\nabla}_{i}(\chi v \sin^2g)_{jk}&=\chi\sin^2\nabla_iv\cdot g_{jk},\qquad
	\wt{\nabla}_{i}(\chi v \sin^2g)_{j\theta}=-\chi\sin\cos v\cdot g_{ij},\\
	\wt{\nabla}_{\theta}(\chi v \sin^2g)_{\theta k}&=\wt{\nabla}_{\theta}(\chi v \sin^2g)_{\theta \theta}=0,
	\end{split}
	\end{equation}
	so that \eqref{tildelemma3} follows from taking the trace and adding up.
\end{proof}
\begin{lem}\label{hatdiv}Let $\omega\in D(M)$, ${\varphi}\in C^{\infty}(0,\infty)$ and $\ol{\omega}\in C^{\infty}(T^*\overline{M})$ be given by $\ol{\omega}=\varphi\cdot r\omega$. Then,
\begin{align}\label{divfree1formsA1}
\ol{\delta}\ol{\omega}=0,\qquad \ol{\omega}(\partial_r)=0,
\end{align}
and
\begin{align}\label{divfree1formsA2}
\ol{\Delta}_1\ol{\omega}=r(-\partial^2_{rr}{\varphi}\cdot\omega-nr^{-1}\partial_r{\varphi}\cdot\omega+{\varphi}\cdot r^{-2}(\Delta_1+1)\omega).
\end{align}
Similarly, if $\wt{\omega}\in D(\widetilde{M})$, ${\varphi}\in C^{\infty}(0,\infty)$ and $\wh{\omega}\in C^{\infty}(T^*\wh{M})$ are given by $\wh{\omega}=\varphi\cdot s\wt{\omega}$, then
\begin{align}\label{divfree1formsA21}
\wh{\delta}\wh{\omega}=0,\qquad \wh{\omega}(\partial_s)=0,
\end{align}
and
\begin{align}\label{divfree1formsA22}
\wh{\Delta}_1\wh{\omega}=s(-\partial^2_{ss}{\varphi}\cdot\wt{\omega}-(n+1)s^{-1}\partial_s{\varphi}\cdot\wt{\omega}+{\varphi}\cdot s^{-2}(\wt{\Delta}_1+1)\wt{\omega}).
\end{align}
	\end{lem}
\begin{proof}
	To prove the lemma, it suffices to show \eqref{divfree1formsA1} and \eqref{divfree1formsA2} as \eqref{divfree1formsA21} and \eqref{divfree1formsA22} follow by relabelling and shifting the dimension by $1$. If $\omega $ and $\ol{\omega}$ are as in the lemma, \eqref{christoffelA2} implies
	\begin{align}
	\ol{\nabla}_i\ol{\omega}_{j}=\varphi r\nabla_i\omega_j,\qquad \ol{\nabla}_r\ol{\omega}_j=\partial_r\varphi\cdot r\omega_j,\qquad 	\ol{\nabla}_i\ol{\omega}_{r}=-\varphi\cdot\omega_j,\qquad \ol{\nabla}_r,\ol{\omega}_{r}=0,
	\end{align}
	and \eqref{divfree1formsA1} follows by taking the trace and the fact that $\ol{\omega}_r=0$. By applying the covariant derivative once again, we obtain from \eqref{christoffelA2} that
	\begin{equation}
	\begin{split}
	\ol{\nabla}_{ij}^2\ol{\omega}_k&=\varphi\cdot r\nabla_{ij}^2\omega_k+r^2g_{ij}\partial_r\varphi\cdot\omega_k-r\varphi g_{jk}\omega_i,\qquad 	\ol{\nabla}_{rr}^2\ol{\omega}_k=r\partial^2_{rr}\varphi\cdot\omega_k,\\
		\ol{\nabla}_{ij}^2\ol{\omega}_r&=-\varphi(\nabla_i\omega_j+\nabla_j\omega_i)=-2\varphi\cdot(\delta^*\omega)_{ij},\qquad
		\ol{\nabla}_{rr}^2\ol{\omega}_r=0,
	\end{split}
	\end{equation}
	and \eqref{divfree1formsA2} follows by taking the trace and using that $\delta \omega=0$.
	\end{proof}
\begin{lem}\label{coupled1forms}
	Let $v\in C^{\infty}(M)$ and $P=P[r,z]$, $Q=Q[r,z]$, $R=R[r,z]$ be functions in two variables, considered as functions on $\wh{M}$. Let $\wh{\omega}=Pdz+Qdr+Rrdv\in C^{\infty}(T^*\wh{M})$. Then,
	\begin{align*}
	\wh{\delta}\wh{\omega}=-\partial_zP-\partial_rQ-nr^{-1}\cdot Q+r^{-1}R\Delta_gv.
	\end{align*}
\end{lem}
\begin{proof}
	The expressions for the Christoffel collected in \eqref{christoffelA2} imply
	\begin{align}
	\wh{\nabla}_z\wh{\omega}_z=\partial_zP,\qquad 	\wh{\nabla}_r\omega_r=\partial_rQ,\qquad \wh{\nabla}_i\omega_j=Qr\cdot g_{ij}+Rr\cdot\nabla^2_{ij}v
	\end{align}
	and the formula follows by taking the trace with respect to $\wh{g}$.
\end{proof}
\begin{lem}\label{gradhessdr}
The $1$-form $dr\in C^{\infty}(T^*\ol{M})$ satisfies
\begin{align}\label{gradhessdr1}
\ol{\nabla}dr=r\cdot g,\qquad \ol{\Delta}dr=nr^{-1}dr.
\end{align}	
Moreover, $v\in C^{\infty}(M)$ satisfies
\begin{align}
\ol{\nabla}(r\cdot dv)=r\nabla dv-dv\otimes dr,\qquad \ol{\Delta}_1(rdv)=r^{-2}(r(\Delta_1+1)dv-2\Delta v\cdot dr).
\end{align}
	\end{lem}
\begin{proof}
	For $dr$, one computes
	\begin{align}
	\ol{\nabla}_idr_j=r\cdot g_{ij},\qquad 	\ol{\nabla}_idr_r=	\ol{\nabla}_rdr_j=	\ol{\nabla}_rdr_r=0,
	\end{align}
	which proves the formula for the gradient of $dr$ and
	\begin{align}
	\ol{\nabla}^2_{ij}dr_r=-g_{ij},\qquad \ol{\nabla}^2_{ij}dr_k=\ol{\nabla}^2_{rr}dr_r=\ol{\nabla}^2_{rr}dr_k=0,
	\end{align}
	from which the formula for the Laplacian follows by taking the trace with respect to $\ol{g}$.
	For $rdv$, we get
	\begin{align}
	\ol{\nabla}_i(rdv)_j=r\nabla_idv_j,\qquad \ol{\nabla}_i(rdv)_r=-dv_i,\qquad \ol{\nabla}_r(rdv)_j=\ol{\nabla}_r(rdv)_r=0,
	\end{align}
		which proves the formula for the gradient of $rdv$ and
		\begin{equation}
		\begin{split}
		\ol{\nabla}^2_{ij}(rdv)_k&=r\nabla^2_{ij}dv_k-rg_{ik}dv_j,\qquad
		\ol{\nabla}^2_{ij}(rdv)_r=-2\nabla_idv_j,\\ \ol{\nabla}^2_{rr}(rdv)_k&=\ol{\nabla}^2_{rr}(rdv)_r=0,
		\end{split}
		\end{equation}
	from which the formula for the Laplacian follows by taking the trace with respect to $\ol{g}$.		
\end{proof}

\begin{lem}\label{ttlemma}
	Let $h\in TT(M)$, $\varphi\in C^{\infty}((0,\infty))$
	 and $\ol{h}\in C^{\infty}(S^2\overline{M})$ be defined by $\ol{h}=\varphi\cdot r^2h$. This tensor satisfies
	\begin{align}
	\ol{\trace}\ol{h}=0,\qquad \ol{\delta}\ol{h}=0,\qquad \ol{h}(\partial_r,.)=0,
	\end{align}
   and
	\begin{align}
	\overline{\Delta}_E\ol{h}=r^2(-\partial^2_{rr}\varphi\cdot h-n\cdot r^{-1}\partial_r\varphi\cdot h+\varphi\cdot r^{-2}\Delta_Eh).
	\end{align}
		Similarly, if $\wt{h}\in TT(\wt{M})$, $\varphi\in C^{\infty}((0,\infty))$
	and $\widehat{h}\in C^{\infty}(S^2\widehat{M})$ is defined by $\widehat{h}=\varphi\cdot r^2\wt{h}$, then
	\begin{align}
	\wh{\trace}\wh{h}=0,\qquad \wh{\delta}\wh{h}=0,\qquad \wh{h}(\partial_r,.)=0,
	\end{align}
	and
	\begin{align}
	\wh{\Delta}_E\wh{h}=s^2(-\partial^2_{ss}\varphi\cdot \wt{h}-(n+1)\cdot s^{-1}\partial_s\varphi\cdot \wt{h}+\varphi\cdot s^{-2}\wt{\Delta}_E\wt{h}).
	\end{align}
\end{lem}
\begin{proof}It suffices to prove the first part of the assertion as the second one follows by relabelling and shifting dimension by one.
	At first $\ol{\trace}\ol{h}=0$ is immediate because $\trace_gh=0$ and $\ol{h}(\partial_r,.)=0$ holds as $\ol{h}_{rr}=\ol{h}_{rj}=0$. By using \eqref{christoffelA2},
	\begin{equation}
	\begin{split}
	\ol{\nabla}_i\ol{h}_{jk}&=\varphi r^2\nabla_ih_{jk},\qquad \ol{\nabla}_rh_{ij}=\partial_r\varphi\cdot r^2h_{ij},\qquad \ol{\nabla}_i\ol{h}_{jr}=\ol{\nabla}_i\ol{h}_{rj}=-\varphi\cdot r\cdot h_{ij},\\
	\ol{\nabla}_i\ol{h}_{rr}&=\ol{\nabla}_r\ol{h}_{jr}=\ol{\nabla}_r\ol{h}_{rk}=0,
	\end{split}
	\end{equation}
	and by taking the trace with respect to $\overline{g}$ and using $\trace_gh=0$, we obtain $ \ol{\delta}\ol{h}=0$. Taking the covariant derivative once again, we obtain
	\begin{equation}
\begin{split}
	\ol{\nabla}^2_{ij}\ol{h}_{kl}&=\varphi r^2\cdot \nabla^2_{ij}h_{kl}-rg_{ij}\partial_r\varphi\cdot r^2h_{kl}+r^2(g_{ik}h_{jl}+g_{il}h_{jk}),\\
	\ol{\nabla}^2_{rr}\ol{h}_{kl}&=\partial^2_{rr}\varphi \cdot r^2h_{kl},\\
	\ol{\nabla}^2_{ij}\ol{h}_{rr}&=	2\varphi\cdot h_{ij},\\
	\ol{\nabla}^2_{rr}\ol{h}_{rr}&=	\ol{\nabla}^2_{rr}\ol{h}_{kr}=	\ol{\nabla}^2_{rr}\ol{h}_{rl}=	0,\\
	\ol{\nabla}^2_{ij}\ol{h}_{kr}&=	\ol{\nabla}^2_{ij}\ol{h}_{rk}=-2\varphi r(\nabla_ih_{jk}+\nabla_jh_{ik}).
	\end{split}
\end{equation}
	By taking the trace and using that $h\in TT(M)$, we obtain
	\begin{align}
	\overline{\Delta}_2\ol{h}=r^2(-\partial^2_{rr}\varphi\cdot h-n\cdot r^{-1}\partial_r\varphi\cdot h+\varphi\cdot r^{-2}\Delta_2h).	
	\end{align}
	It remains to consider the curvature term. However, the only nonvanishing term of the curvature of $\overline{g}$ is
	\begin{align}
	\overline{R}_{ijkl}=r^2(R_{ijkl}+g_{ik}g_{jl}-g_{il}g_{jk}),
	\end{align}
	so that 
	\begin{align}
	\mathring{\overline{R}}(\overline{h})_{ij}=\varphi\mathring{R}(h)_{ij},
	\end{align}
	which by adding up finishes the proof of the lemma.
\end{proof}
\begin{lem}
	Let $\omega\in D(M)$ and $P=P[r,z]$, $Q=Q[r,z]$, $R=R[r,z]$ be functions in two variables, considered as functions on $\wh{M}$. Let 
	\begin{align}\wh{h}=Pdz\odot r\omega+Qdr\odot r\omega+Rr^2\delta^*\omega\in \Gamma(S^2\wh{M}).
	\end{align}
	 Then,
	\begin{align}
	\wh{\delta}\wh{h}=-(\partial_zP+\partial_rQ)r\omega -(n+1)Q\cdot r\omega+\frac{1}{2}R(\Delta_1\omega-(n-1)\omega).
	\end{align}
\end{lem}
\begin{proof}
	Straightforward calculations show that
	\begin{equation}
\begin{split}
	\wh{\nabla}_z\wh{h}_{zz}&=\wh{\nabla}_r\wh{h}_{rz}=0,\qquad \wh{\nabla}_i\wh{h}_{jz}=P\cdot r\cdot \nabla_i\omega_j,\\
	\wh{\nabla}_z\wh{h}_{zr}&=\wh{\nabla}_r\wh{h}_{rr}=0,\qquad \wh{\nabla}_i\wh{h}_{jr}=Q\cdot r\cdot \nabla_i\omega_j-Rr\cdot \delta^*\omega_{ij},\\
	\wh{\nabla}_z\wh{h}_{zk}&=\partial_zP \cdot r\cdot \omega_k,\qquad 	\wh{\nabla}_r\wh{h}_{rk}=\partial_r Q \cdot r\cdot \omega_k,\qquad
	\wh{\nabla}_i\wh{h}_{jk}=\nabla_i(\delta^*\omega)_{jk}+(n+1)Q\omega_k,
	\end{split}
\end{equation}
	and the proof follows from taking the trace and using the identity
	\begin{align}
	\delta\delta^*\omega=\frac{1}{2}(\Delta_1\omega-(n-1)\omega),
	\end{align}
	cf.\ Remark \ref{Killing}.
\end{proof}
\begin{lem}\label{philemma}
	Let $\omega\in D(M)$. Then,
	\begin{align}\label{dm1}
    \ol{\Delta}_2(dr\odot r\omega)&=-4\delta^*\omega+(\mu+n+3) dr\odot r^{-1}\omega,\\
 \label{dm2} \ol{\delta}^*\omega&={\delta}^*\omega-r^{-2}dr\odot r\omega .   
	\end{align}
\end{lem}
\begin{proof}
	By using \eqref{christoffelA2}, we compute
	\begin{equation}
\begin{split}
	\ol{\nabla}_i(dr\odot r\omega)_{jk}&=r^2(\omega_j g_{ik}+\omega_kg_{ij}),\qquad
	\ol{\nabla}_i(dr\odot r\omega)_{rr}=-2\omega_i,\\
	\ol{\nabla}_i(dr\odot r\omega)_{jr}&=\ol{\nabla}_i(dr\odot r\omega)_{rj}=r\nabla_i\omega_j,
	\end{split}
\end{equation}
	while the other components vanish. The second covariant derivative is computed as follows:
	\begin{equation}
\begin{split}
	\ol{\nabla}_{ij}(dr\odot r\omega)_{kl}&=r^2(\nabla_i\omega_k\cdot g_{jl}+\nabla_i\omega_l\cdot g_{jk}+\nabla_j\omega_k\cdot g_{il}+\nabla_j\omega_l\cdot g_{ik}),\\
	\ol{\nabla}_{rr}(dr\odot r\omega)_{kl}&=\ol{\nabla}_{ij}(dr\odot r\omega)_{rr}=0,\qquad
	\ol{\nabla}_{ij}(dr\odot r\omega)_{rr}=-2(\nabla_i\omega_j+\nabla_j\omega_i),\\
	\ol{\nabla}_{ij}(dr\odot r\omega)_{kr}&=\ol{\nabla}_{ij}(dr\odot r\omega)_{rk}=r(\nabla^2_{il}\omega_k-2\omega_j\cdot g_{ik}-\omega_k\cdot g_{ij}-\omega_i\cdot g_{jk}),\\
	\ol{\nabla}_{rr}(dr\odot r\omega)_{kr}&=\ol{\nabla}_{rr}(dr\odot r\omega)_{rl}=0,
	\end{split}
\end{equation}
	and by taking the trace with respect to $\ol{g}$, we obtain \eqref{dm1}.
	To prove \eqref{dm2}, it suffices to see that
	\begin{align}
	\ol{\nabla}_i\omega_j=\nabla_i\omega_j,\qquad 	\ol{\nabla}_i\omega_r=	\ol{\nabla}_r\omega_i=-r^{-1}\omega_i,\qquad \ol{\nabla}_r\omega_r=0,
	\end{align}
	and to use the definitions of $\ol{\delta}^*$ and $\delta^*$.
	\end{proof}
\begin{lem}
	Let $v\in C^{\infty}(M)$,  $P^{(i)},Q^{(j)},R,S$ smooth functions in two variables for $i=1,2,3$ and $j=1,2$ and define $\wh{h}\in C^{\infty}(S^2\wh{M})$ as
	\begin{equation}
\begin{split}
\wh{h} &=P^{(1)}[r,z]vdz\otimes dz+P^{(2)}[r,z]vdz\odot dr+P^{(3)}[r,z]vdr\otimes dr+Q^{(1)}[r,z]dz\odot rdv\\
&\qquad+Q^{(2)}[r,z]dr\odot rdv+R[r,z]r^2(n\nabla^2v+\Delta v g)+S[r,z]vr^2g.
	\end{split}
\end{equation}
Then the divergence of $\wh{h}$ is computed as
	\begin{equation}
\begin{split}
\wh{\delta}\wh{h}&=(-\partial_zP^{(1)}v-\partial_r P^{(2)}v+r^{-1}Q^{(1)}\Delta v-r^{-1}n P^{(2)}v)dz\\
&\qquad+ (-\partial_zP^{(2)}v-\partial_r P^{(3)}v+r^{-1}Q^{(2)}\Delta v-r^{-1}n P^{(3)}v+r^{1}nS)dr\\
&\qquad +(-r\cdot\partial_zQ^{(1)}-r\cdot\partial_rQ^{(2)}-(n+1)Q^{(2)}-S)dv+R(n-1)d(\Delta v-nv).
	\end{split}
\end{equation}
\end{lem}
\begin{proof}
	We compute
	\begin{equation}
\begin{split}
	\wh{\nabla}_z\wh{h}_{zz}&=\partial_zP^{(1)}v,\qquad\wh{\nabla}_r\wh{h}_{rz}=\partial_rP^{(2)}v,\qquad 
	\wh{\nabla}_i\wh{h}_{jz}=Q^{(1)}r\nabla^2_{ij}v+r\cdot g_{ij}P^{(2)}\cdot v,\\
	\wh{\nabla}_z\wh{h}_{zr}&=\partial_zP^{(2)}v,\qquad\wh{\nabla}_r\wh{h}_{rz}=\partial_rP^{(3)}v,\qquad 
	\wh{\nabla}_i\wh{h}_{jr}=Q^{(2)}r\nabla^2_{ij}v-\frac{1}{r}\wh{h}_{ij}+r\cdot g_{ij}P^{(3)}v.
	\end{split}
\end{equation}
	By taking the trace with respect to $\wh{g}$, we obtain the $z$ and the $r$-component of the divergence. To compute the $dv$- component, we calculate
	\begin{equation}
\begin{split}
	\wh{\nabla}_z\wh{h}_{zk}&=\partial_zQ^{(1)}rdv_k,\qquad\wh{\nabla}_r\wh{h}_{rk}=\partial_rQ^{(2)}rdv_k,\\
	\wh{\nabla}_i\wh{h}_{jk}&=\nabla_i\wh{h}_{jk}+r^2g_{ij}Q^{(2)}dv_k+r^2g_{ik}Q^{(2)}dv_j.	
	\end{split}
\end{equation}
	On the right hand side of the last equation, $\wh{h}_{jk}$ is understood as a parameter-dependent section of $S^2M$ and the covariant derivative has to be understood in that way. We have,
	\begin{align}
	\delta^*\nabla^2v=d(\Delta v-(n-1)v),\qquad \delta(\Delta v g)=-d(\Delta v),\qquad \delta(vr^2g)=-r^2 dv,
	\end{align}
	where the first equation follows from \eqref{commutation} and the other identities are standard.
 Using these identities and taking the trace with respect to $\wh{g}$ yields the last component of the divergence.
	\end{proof}
\begin{lem}
	Consider the $1$-form $dr\in C^{\infty}(T^*\ol{M})$ and $v\in C^{\infty}(M)$. Then we have the identities
	\begin{align}
	\label{lem1}\ol{\Delta}(dr\otimes dr)&=2r^{-2}(ndr\otimes dr-r^2g),\\
	\label{lem2}\ol{\Delta}(r^2g)&=2r^{-2}(r^2g-ndr\otimes dr),\\
	\label{lem3}\ol{\Delta}(dr\odot rdv)&=-4r^{-2}\Delta v\cdot dr\otimes dr+r^{-2}dr\odot rd(\Delta v+4v)-4\nabla^2v,\\
	\label{lem4}\ol{\nabla}^2(r^2v)&=r^2\nabla^2v+2v\ol{g}+dr\otimes rdv,\\
	\label{lem5}\ol{\nabla}^2v&=\nabla^2v-r^{-2}dr\odot rdv,\\ 
	\label{lem6}\ol{\Delta}(r^2v)&=(\Delta-2n-2)v.
	\end{align}
	\end{lem}
\begin{proof}
	Equation \eqref{lem1} follows from \eqref{gradhessdr1}, since
	\begin{align}
	\ol{\Delta}_2(dr\otimes dr)=(\ol{\Delta}dr)\otimes dr+dr\otimes(\ol{\Delta} dr)-2\langle\ol{\nabla}dr,\ol{\nabla}dr\rangle_{\ol{g}}=2r^{-2}(ndr\otimes dr-r^2g),
	\end{align}
	and \eqref{lem2} follows from the fact that
	\begin{align}
	0=\ol{\Delta}_2(\ol{g})=\ol{\Delta}_2(dr\otimes dr)+\ol{\Delta}_2(r^2g).
	\end{align}
	To prove \eqref{lem3}, we compute
	\begin{equation}
\begin{split}
    \ol{\nabla}_i(dr\odot rdv)_{jk}&=r^2g_{ij}dv_k,\qquad  \ol{\nabla}_i(dr\odot rdv)_{jr}= \ol{\nabla}_i(dr\odot rdv)_{rj}=r\nabla_idv_j,\\
     \ol{\nabla}_i(dr\odot rdv)_{rr}&=-2dv_i,\\ \ol{\nabla}_r(dr\odot rdv)_{jk}&=\ol{\nabla}_r(dr\odot rdv)_{rr}=\ol{\nabla}_r(dr\odot rdv)_{ir}=\ol{\nabla}_r(dr\odot rdv)_{ri}=0.
	\end{split}
\end{equation}
	By taking the covariant derivative once again, we obtain
	\begin{equation}
\begin{split}
	\ol{\nabla}^2_{ij}(dr\odot rdv)_{kl}&=r^2(g_{jk}\nabla_idv_l+g_{jl}\nabla_idv_k+g_{ik}\nabla_jdv_l+g_{il}\nabla_jdv_k),\\
	\ol{\nabla}^2_{ij}(dr\odot rdv)_{k0}&=\ol{\nabla}^2_{ij}(dr\odot rdv)_{0k}=r\nabla^2_{ij}dv_k-2r(g_{ik}dv_j+g_{jk}dv_i+g_{ij}dv_k),\\
	\ol{\nabla}^2_{rr}(dr\odot rdv)_{rr}&=\ol{\nabla}^2_{rr}(dr\odot rdv)_{rl}=\ol{\nabla}^2_{rr}(dr\odot rdv)_{kl}=0,
		\end{split}
	\end{equation}
	and \eqref{lem3} follows by taking the trace with respect to $\ol{g}$ and using $\Delta_1\circ d=d\circ (\Delta-(n-1))$.
	For the proof of \eqref{lem4}, we calculate
	\begin{align}
	\ol{\nabla}^2_{ij}(r^2v)=r^2(\nabla^2_{ij}v+2g_{ij}v),\qquad \ol{\nabla}^2_{rr}(r^2v)=2v,\qquad \ol{\nabla}^2_{ri}(r^2v)=\ol{\nabla}^2_{ir}(r^2v)=r\partial_iv,
	\end{align}
	whereas \eqref{lem5} follows from
	\begin{align}
	\ol{\nabla}^2_{ij}v=\nabla^2_{ij}v,\qquad \ol{\nabla}^2_{ri}v=\ol{\nabla}^2_{ir}v=-r^{-1}\partial_iv,\qquad \ol{\nabla}^2_{rr}v=0.
	\end{align}
	Finally, \eqref{lem6} is a direct application of \eqref{olDelta}.
\end{proof}
\section{Verifying PDE's in two variables}
In this section, we consider functions in two variables and the operator
\begin{align}
\wh{\Delta}=-\partial^2_{zz}-\partial^2_{rr}-n r^{-1}\partial_r.
\end{align}
Consider the vector field $V=r\partial_z-z\partial_r$. As a differential operator, it satisfies the following commutation relations:
\begin{align}
\label{comm_1}[\partial_V,\wh{\Delta}]&=n r^{-2}\cdot \partial_V,\\
\label{comm_2}[\partial_V,r^{-2}]&=2z\cdot r^{-3},\\
\label{comm_3}[\partial_V,z r^{-1}]&=1+z^2\cdot r^{-2}.
\end{align}
Let $f=f(z,r)$ be a smooth function and $g=-zr^{-1}\cdot f$. Then for $r\neq0$, we have the equations
\begin{align}
\label{rzformula_1}\wh{\Delta}g&=-z r^{-1}\wh{\Delta}f+(n-2)r^{-2}\cdot g+2r^{-2}\cdot \partial_Vf,\\
\label{rzformula_2}r\partial_z f+r\partial_rg&=\partial_Vf-g.
\end{align}

\begin{lem}\label{formulas1}
	Let $\lambda>0$ and $P=P(z,r)$ be a solution of the equation
	\begin{align}\wh{\Delta}P+\lambda r^{-2}P=0,
\end{align}	
	 and $Q$ and $R$ be defined by 
\begin{align}\label{defQ}Q&=-\frac{z}{r}\cdot P,\\
\label{defR}\lambda\cdot R&=r\partial_z P+r\partial_r Q+n\cdot Q.
\end{align}	
	 Then
\begin{align}
\label{eqQ}\wh{\Delta}Q+r^{-2}(\lambda+n)Q-2r^{-2}\lambda R&=0,\\ 
\label{eq3R}\wh{\Delta}R+r^{-2}(\lambda+2-n)R-2r^{-2}Q&=0.
\end{align}
\end{lem}
\begin{proof}
By \eqref{rzformula_1}, 
\begin{align}
\wh{\Delta}Q&=-z r^{-1}\wh{\Delta}P+(n-2)r^{-2}\cdot Q+2r^{-2}\cdot \partial_VP=(n-\lambda-2)r^{-2}\cdot Q+2r^{-2}\cdot \partial_VP,
\end{align}
and by \eqref{rzformula_2}, \eqref{defR} is equivalent to
\begin{align}\label{newdefR}
\partial_VP=\lambda R-(n-1)Q,
\end{align}
which proves \eqref{eqQ}. To prove \eqref{eq3R}, we use \eqref{newdefR} and \eqref{eqQ} to compute
\begin{align*}
\lambda\wh{\Delta}R&=(n-1)\wh{\Delta}Q+\wh{\Delta}\partial_V P\\
&=(n-1)[-(\lambda+n)r^{-2}Q+2\lambda r^{-2}R]+\partial_V(\wh{\Delta}P)+[\wh{\Delta},\partial_V]P\\
&=-(\lambda+n)(n-1)r^{-2}Q+2\lambda(n-1)r^{-2}R-\lambda\partial_V(r^{-2}P)-nr^{-2}\cdot\partial_VP\\
&=-(\lambda+n)(n-1)r^{-2}Q+2\lambda(n-1)r^{-2}R-(\lambda+n)r^{-2}\partial_VP+\lambda[r^{-2},\partial_V]P\\
&=-(\lambda+n)(n-1)r^{-2}Q+2\lambda(n-1)r^{-2}R-(\lambda+n)r^{-2}(\lambda R-(n-1)Q)+2\lambda r^{-2}Q\\
&=2\lambda r^{-2}Q-\lambda(\lambda+2-n)r^{-2}R,
\end{align*}
which yields \eqref{eq3R}.
\end{proof}
\begin{lem}\label{formulas2}
	Let $\mu>n-1$ and $P=P(z,r)$ be a solution of the equation
	\begin{align}\wh{\Delta}P+r^{-2}(\mu+1)P&=0,
\end{align}	
	 and $Q$ and $R$ be defined by 
\begin{align}\label{def2Q}Q&=-\frac{z}{r}\cdot P,\\
\label{def2R}\frac{1}{2}(\mu-(n-1))R&=r\cdot\partial_zP+r\cdot \partial_rQ+(n+1)Q.
\end{align}	
	 Then
\begin{align}
\label{eq2Q}\wh{\Delta}Q+r^{-2}(\mu+n+3)Q-r^{-2}(\mu+1-n)R&=0,\\ 
\label{eq2R}\wh{\Delta}R+r^{-2}(\mu+1-n)R-4r^{-2}Q&=0.
\end{align}
\end{lem}
\begin{proof}
By \eqref{rzformula_1}, 
\begin{align}
\wh{\Delta}Q&=-z r^{-1}\wh{\Delta}P+(n-2)r^{-2}\cdot Q+2r^{-2}\cdot \partial_VP=(n-\mu-3)r^{-2}\cdot Q+2r^{-2}\cdot \partial_VP,
\end{align}
and by \eqref{rzformula_2}, \eqref{def2R} is equivalent to
\begin{align}
\partial_VP=\frac{1}{2}(\mu+1-n) R-nQ,
\end{align}
which proves \eqref{eq2Q}. To prove \eqref{eqR}, we use \eqref{newdefR} and \eqref{eqQ} to compute
\begin{align*}
\frac{1}{2}(\mu+1-n)\wh{\Delta}R&=\wh{\Delta}\partial_VP+n\wh{\Delta}Q\\
&=\partial_V\wh{\Delta}P+[\wh{\Delta},\partial_V]P+n[(\mu+1-n)r^{-2}R-(\mu+n+3)r^{-2}Q]\\
&=-\partial_V((\mu+1)r^{-2}P)-nr^{-2}\partial_VP+n[(\mu+1-n)r^{-2}R-(\mu+n+3)r^{-2}Q]\\
&=-(\mu+1)r^{-2}\partial_VP+(\mu+1)[r^{-2},\partial_V]P-nr^{-2}\partial_VP\\
&\qquad+n[(\mu+1-n)r^{-2}R-(\mu+n+3)r^{-2}Q]\\
&=-(\mu+n+1)r^{-2}\partial_VP-2(\mu+1)z r^{-3}P\\
&\qquad+n[(\mu+1-n)r^{-2}R-(\mu+n+3)r^{-2}Q]\\
&=-(\mu+n+1)r^{-2}(\frac{1}{2}(\mu+1-n)R-nQ)+2(\mu+1)r^{-2}Q\\
&\qquad+n[(\mu+1-n)r^{-2}R-(\mu+n+3)r^{-2}Q]\\
&=\frac{1}{2}(\mu+1-n)[(n-\mu-1)r^{-2}R+4r^{-2}Q],
\end{align*}
which yields \eqref{eq2R}.
\end{proof}
\begin{lem}\label{formulas3}
	Let $\lambda>n$ and
	 $P^{(1)}=P^{(1)}(z,r)$ be a solution of the equation
	\begin{align}\wh{\Delta}P^{(1)}+\lambda r^{-2} P^{(1)}&=0,
\end{align}	
	 and $P^{(2)},P^{(3)},Q^{(1)},Q^{(2)},R$ and $S$ be defined by 
\begin{align}\label{defP^{(2)}}	P^{(2)}&=-\frac{z}{r}P^{(1)},\\
\label{defP^{(3)}} P^{(3)}&=-\frac{z}{r}P^{(2)}=\frac{z^2}{r^2}P^{(1)},\\
\label{defQ^{(2)}}  Q^{(2)}&=-\frac{z}{r}Q^{(1)},\\
\label{defS}  nS&=-P^{(1)}-	P^{(3)},\\
\label{defQ^{(1)}}\lambda Q^{(1)}&=r\cdot\partial_zP^{(1)}+r\cdot\partial_rP^{(2)}+nP^{(2)},\\
\label{def2Q^{(2)}}\lambda Q^{(2)}&=r\cdot\partial_zP^{(2)}+r\cdot\partial_rP^{(3)}+nP^{(3)}-n S,\\
\label{def3R}(n-1)(\lambda-n) R&=r\cdot\partial_zQ^{(1)}+r\cdot\partial_rQ^{(2)}+(n+1)Q^{(2)}+S.
\end{align}	
	 Then,
\begin{align}
\label{eqP^{(1)}}\wh{\Delta}P^{(1)}+\lambda r^{-2}P^{(1)}&=0,\\
\label{eqP^{(2)}}\wh{\Delta}P^{(2)}+(\lambda+n)r^{-2}P^{(2)}-2r^{-2}\lambda Q^{(1)}&=0,\\
\label{eqQ^{(1)}}\wh{\Delta}Q^{(1)}+(\lambda-n+2)r^{-2}Q^{(1)}-2r^{-2}P^{(2)}&=0,\\
\label{eqP^{(3)}}\wh{\Delta}P^{(3)}+(\lambda+2n)r^{-2}P^{(3)}-2nSr^{-2}-4\lambda r^{-2}Q^{(2)}&=0,\\
\label{eq2S}\wh{\Delta}S+(\lambda+2)r^{-2}S-2r^{-2}P^{(3)}+\frac{4}{n}r^{-2}\lambda Q^{(2)}&=0,\\
\label{eqQ^{(2)}}\wh{\Delta}Q^{(2)}+(\lambda+4)r^{-2}Q^{(2)}-2r^{-2}P^{(3)}+2r^{-2}S+2(n-1)(n-\lambda)r^{-2}R&=0,\\
\label{eq4R}\wh{\Delta}R+(\lambda-2n+2)r^{-2}R-\frac{4}{n}r^{-2}Q^{(2)}&=0.
\end{align}
\end{lem}
\begin{proof}
	By \eqref{rzformula_2}, \eqref{defQ^{(1)}}-\eqref{def3R} are equivalent to
	\begin{align}
\label{newdefQ^{(1)}}	\partial_VP^{(1)}&=\lambda\cdot Q^{(1)}-(n-1)P^{(2)},\\
\label{newdef2Q^{(2)}}	\partial_VP^{(2)}&=\lambda\cdot Q^{(2)}-(n-1)P^{(3)}+nS,\\
\label{newdef3R}	\partial_VQ^{(1)}&=(n-1)(\lambda-n)R-nQ^{(2)}-S,
	\end{align}
	respectively. Let us first check that \eqref{defQ^{(2)}} and\eqref{newdefQ^{(1)}} are consistent with \eqref{newdef2Q^{(2)}}.
	By multiplying \eqref{newdefQ^{(1)}} with $-z r^{-1}$ and using \eqref{defP^{(3)}} and \eqref{defQ^{(2)}}, we get
	\begin{align*}
	\lambda\cdot Q^{(2)}-(n-1)P^{(3)}=-z r^{-1}\partial_VP^{(1)}&=\partial_V(-z r^{-1}P^{(1)})+[V,z r^{-1}](P^{(1)})\\
	&=\partial_VP^{(2)}+(1+z^2\cdot r^{-2})P^{(1)}\\
	&=\partial_VP^{(2)}+P^{(1)}+P^{(3)}=\partial_VP^{(2)}-nS,
	\end{align*}
	which yields \eqref{newdef2Q^{(2)}}. Now, \eqref{eqP^{(1)}}-\eqref{eqQ^{(1)}} follow already from Lemma \ref{formulas1}, as the relation between $P^{(1)},P^{(2)}$ and $Q^{(1)}$ is the same as between $P,Q$ and $R$ in Lemma \ref{formulas1}. To prove \eqref{eqP^{(3)}}, we use \eqref{rzformula_1}, \eqref{eqP^{(2)}} and \eqref{newdef2Q^{(2)}} to compute
	\begin{align*}
	\wh{\Delta}P^{(3)}&=-z r^{-1}\wh{\Delta}P^{(2)}+(n-2)r^{-2}P^{(3)}+2r^{-2}\partial_VP^{(2)}\\
	&=-z r^{-1}(-(\lambda+n)r^{-2}P^{(2)}+2\lambda r^{-2}Q^{(1)})+(n-2)r^{-2}P^{(3)}\\
	&\qquad+2r^{-2}(\lambda\cdot Q^{(2)}-(n-1)P^{(3)}+nS)\\
	&=-(\lambda+2n)r^{-2}P^{(3)}+4\lambda r^{-2}Q^{(2)}+2nr^{-2}S,
	\end{align*}
	where we also used \eqref{defP^{(3)}} and \eqref{defQ^{(2)}} in the last equality. To get \eqref{eq2S}, we apply $\wh{\Delta}$ to \eqref{defS} and use \eqref{eqP^{(1)}} and \eqref{eqP^{(3)}} which yields
	\begin{align*}
	n\wh{\Delta}S&=-\wh{\Delta}P^{(1)}-\wh{\Delta}P^{(3)}
	\\&=\lambda r^{-2}P^{(1)}+(\lambda+2n)r^{-2}P^{(3)}-4\lambda r^{-2}Q^{(2)}-2nr^{-2}S\\
	&=-n(\lambda+2)r^{-2}S+2nr^{-2}P^{(3)}-4\lambda r^{-2}Q^{(2)},
	\end{align*}
	where we also used \eqref{defS} in the last step. For the computation of \eqref{eqQ^{(2)}}, we apply $\wh{\Delta}$ to \eqref{newdef2Q^{(2)}} to get
	\begin{align*}
	\lambda\wh{\Delta}Q^{(2)}&=\wh{\Delta}\partial_VP^{(2)}+n\wh{\Delta}P^{(3)}+\wh{\Delta}P^{(1)}\\
	&=\partial_V(\wh{\Delta}P^{(2)})+[\wh{\Delta},\partial_V]P^{(2)}+n(4\lambda r^{-2}Q^{(2)}+2nr^{-2}S-(\lambda+2n)r^{-2}P^{(3)})-\lambda r^{-2}P^{(1)}\\
	&=\partial_V(-(n+\lambda)r^{-2}P^{(2)}+2\lambda r^{-2}Q^{(1)})-n r^{-2}\partial_VP^{(2)}\\
	&\qquad +n(4\lambda r^{-2}Q^{(2)}+2nr^{-2}S-(\lambda+2n)r^{-2}P^{(3)})-\lambda r^{-2}P^{(1)}\\
	&=[V,r^{-2}](-(n+\lambda)P^{(2)}+2\lambda Q^{(1)})-(n+\lambda)r^{-2}\partial_VP^{(2)}+2\lambda r^{-2}\partial_VQ^{(1)}-nr^{-2}\partial_VP^{(2)}\\
	&\qquad +n(4\lambda r^{-2}Q^{(2)}+2nr^{-2}S-(\lambda+2n)r^{-2}P^{(3)})-\lambda r^{-2}(P^{(3)}+nS)\\
	&=2zr^{-3}(-(n+\lambda)P^{(2)}+2\lambda Q^{(1)})-(2n+\lambda)r^{-2}\partial_VP^{(2)}+2\lambda r^{-2}\partial_VQ^{(1)}\\
	&\qquad +n(4\lambda r^{-2}Q^{(2)}+2nr^{-2}S-(\lambda+2n)r^{-2}P^{(3)})-\lambda r^{-2}(P^{(3)}+nS)\\
	&=2(n+\lambda)r^{-2}P^{(3)}-4\lambda r^{-2}Q^{(2)}-(2n+\lambda)r^{-2}(\lambda Q^{(2)}-(n-1)P^{(3)}+nS)\\
	&\qquad +2\lambda r^{-2}((n-1)(\lambda-n)R-nQ^{(2)}-S)\\
	&\qquad+n(4\lambda r^{-2}Q^{(2)}+2nr^{-2}S-(\lambda+2n)r^{-2}P^{(3)})-\lambda r^{-2}(P^{(3)}+nS)\\
	&=-\lambda(\lambda+4)r^{-2}Q^{(2)}+2\lambda r^{-2}P^{(3)}-2\lambda r^{-2}S+2\lambda(n-1)(\lambda-n)r^{-2}R,
	\end{align*}
	where we used the commutation relations \eqref{comm_1},\eqref{comm_2}, the definitions \eqref{defP^{(3)}},\eqref{defQ^{(2)}} and \eqref{defS}, the reformulated definitions \eqref{newdefQ^{(1)}} and \eqref{newdef3R} and the equations \eqref{eqP^{(1)}},\eqref{eqP^{(2)}} and \eqref{eqP^{(3)}}.
	Finally, to prove \eqref{eq4R}, we apply $\wh{\Delta}$ to \eqref{newdef3R} to obtain
	\begin{align*}
	(n-1)(\lambda-n)\wh{\Delta}R&=\wh{\Delta}\partial_VQ^{(1)}+n\wh{\Delta}Q^{(2)}+\wh{\Delta}S\\
	&=\partial_V(\wh{\Delta}Q^{(1)})+[\wh{\Delta},\partial_V]Q^{(1)}+n\wh{\Delta}Q^{(2)}+\wh{\Delta}S\\
	&=\partial_V(2r^{-2}P^{(2)}-(\lambda+2-n)r^{-2}Q^{(1)})-nr^{-2}\partial_VQ^{(1)}+n\wh{\Delta}Q^{(2)}+\wh{\Delta}S\\
	&=[V,r^{-2}](2P^{(2)}-(\lambda+2-n)Q^{(1)})+2r^{-2}\partial_VP^{(2)}\\
	&\qquad-(\lambda+2)\partial_VQ^{(1)}+n\wh{\Delta}Q^{(2)}+\wh{\Delta}S\\
	&=2zr^{-3}(2P^{(2)}-(\lambda+2-n)Q^{(1)})+2r^{-2}\partial_VP^{(2)}\\
	&\qquad-(\lambda+2)\partial_VQ^{(1)}+n\wh{\Delta}Q^{(2)}+\wh{\Delta}S\\
	&=-4r^{-2}P^{(3)}ü2(\lambda+2-n)r^{-2}Q^{(2)}+2r^{-2}(\lambda Q^{(2)}-(n-1)P^{(3)}+nS)\\
	&\qquad -(\lambda+2)r^{-2}((n-1)(\lambda-n)R-nQ^{(2)}-S)\\
	&\qquad +n(-(\lambda+4)r^{-2}Q^{(2)}+2r^{-2}P^{(3)}-2r^{-2}S+2(n-1)(\lambda-n)r^{-2}R)\\
	&\qquad -(\lambda+2)r^{-2}S+2 r^{-2}P^{(3)}-4 n^{-1}\lambda r^{-2}Q^{(2)}\\
	&=(n-1)(\lambda-n)[(2n-\lambda-2)r^{-2}R+4n^{-1}r^{-2}Q^{(2)}],
	\end{align*}
		where we used the commutation relations \eqref{comm_1},\eqref{comm_2}, the definitions \eqref{defP^{(3)}} and \eqref{defQ^{(2)}}, the reformulated definitions \eqref{newdef2Q^{(2)}} and \eqref{newdef3R} and the equations \eqref{eqQ^{(1)}},\eqref{eq2S} and \eqref{eqQ^{(2)}}.
\end{proof}


\begin{thebibliography}{DWW07}


\providecommand{\url}[1]{\texttt{#1}}
\expandafter\ifx\csname urlstyle\endcsname\relax
  \providecommand{\doi}[1]{doi: #1}\else
  \providecommand{\doi}{doi: \begingroup \urlstyle{rm}\Url}\fi


\bibitem[Bam14]{Bam14}
\textsc{Bamler}, Richard:
\newblock{Stability of hyperbolic manifolds with cusps under Ricci flow.}
\newblock {In: }\emph{Adv. Math.} \textbf{263} (2014), 412--467


\bibitem[Bam15]{Bam15}
\textsc{Bamler}, Richard:
\newblock{Stability of symmetric spaces of noncompact type under Ricci flow.}
\newblock {In: }\emph{Geom. Funct. Anal.} \textbf{25} (2015), no 2, 342--416






\bibitem[BGM71]{BGM71}
\textsc{Berger}, Marcel ; \textsc{Gauduchon}, Paul  ; \textsc{Mazet}, Edmond:
\newblock \emph{{Le spectre d'une vari\'et\'e riemannienne.}}
\newblock Lecture Notes in Mathematics, 194, Springer-Verlag, Berlin-Heidelberg, 1971

\bibitem[Bes08]{Bes08}
\textsc{Besse}, Arthur~L.:
\newblock \emph{{Einstein manifolds. Reprint of the 1987 edition.}}
\newblock {Berlin: Springer}, 2008

\bibitem[B\"{o}h98]{Boh98}
\textsc{B\"{o}hm}, Christoph: 
\newblock {Inhomogeneous Einstein metrics on low-dimensional spheres and other low-dimensional spaces.}
\newblock {In: } \emph{Invent. Math.} \textbf{134} (1998), no. 1, 145--176

\bibitem[BS87]{BS87}
\textsc{Br\"{u}ning}, Jochen, and \textsc{Seeley}, Robert:
\newblock {The resolvent expansion for second order regular singular operators.}
\newblock {In: } J. Funct. Anal \textbf{73} (1987), no. 2, 369--429.

\bibitem[BILPS14]{BILPS14}
\textsc{Bunk} Severin ; \textsc{Ivanova}. Tatjana M. ; \textsc{Lechtenfeld}, Olaf ; \textsc{Popov}, Alexander D. ; \textsc{Sperling}, Marcus:
\newblock {Instantons on sine-cones over Sasakian manifolds.}
\newblock {In: }\emph{Phys. Rev. D} \textbf{90} (2014), no. 6, 065028

%




\bibitem[CH15]{CH15}
\textsc{Cao}, Huai-Dong ; \textsc{He}, Chenxu:
\newblock Linear {S}tability of {P}erelmans $\nu$-entropy on {S}ymmetric spaces
  of compact type.
\newblock  In: \emph{J. Reine Angew. Math.} \textbf{709} (2015), no. 5, 229--246

\bibitem[Che83]{Che83}
\textsc{Cheeger}, Jeff: 
\newblock Spectral geometry of singular Riemannian spaces.
\newblock In: \emph{J. Differ. Geom.} \textbf{18} (1983), no.\ 4, 575--657

%

%
%
%
%

\bibitem[DK21]{DK20}
\textsc{Deruelle}, Alix ; \textsc{Kr\"{o}ncke}, Klaus:
\newblock {Stability of ALE Ricci-flat manifolds under Ricci flow.}
\newblock {In: }\emph{J. Geom. Anal.} \textbf{31} (2021),  no. 3, 2829--2870 
%
%



\bibitem[GLNP11]{GLNP11}
\textsc{Gemmer} Karl-Philip ; \textsc{Lechtenfeld}, Olaf ; \textsc{N\"{o}lle}. Christoph ; \textsc{Popov}, Alexander D. :
\newblock{Yang-Mills instantons on cones and sine-cones over nearly Kähler manifolds.}
\newblock{In: }\emph{J. High Energ. Phys.} \textbf{103} (2011), no. 9
%

\bibitem[GHP03]{GHP03}
\textsc{Gibbons}, Gary~W. ; \textsc{Hartnoll}, Sean~A.  ; \textsc{Pope},
  Christopher~N.:
\newblock Bohm and {E}instein-{S}asaki {M}etrics, {B}lack {H}oles, and
  {C}osmological {E}vent {H}orizons.
\newblock {In: }\emph{Phys. Rev. D} \textbf{67} (2003), no. 8







%






%
\bibitem[Kr\"o15]{Kro15b}
\textsc{Kr\"oncke}, Klaus:
\newblock {On infinitesimal Einstein deformations}
\newblock {In: } \emph{Diff. Geom. Appl.} \textbf{38} (2015), no. 1-2, 41--57  

\bibitem[Kr\"o16]{Kro16}
\textsc{Kr\"oncke}, Klaus:
\newblock {Variational Stability and Rigidity of Compact Einstein Manifolds}
\newblock {In: F. Finster et al. (eds.)}
\emph{ Quantum Mathematical Physics. A Bridge between Mathematics and Physics}, 497--513, Springer, 2016 

\bibitem[Kr\"o17]{Kro15c}
\textsc{Kr\"oncke}, Klaus:
\newblock {Stable and unstable Einstein warped products.}
\newblock {In:} \emph{Trans. Amer. Math. Soc.} \textbf{369} (2017), no. 9, 6537--6563 

\bibitem[Kr\"o18]{Kro18}
\textsc{Kr\"oncke}, Klaus:
Stability of sin-cones and cosh-cylinders 
\newblock {In:} \emph{Ann. Sc. Norm. Super. Pisa, Cl. Sci.} \textbf{18} (2018) no. 3, 1155-1187  

\bibitem[Kr\"o20]{Kro13}
\textsc{Kr\"oncke}, Klaus:
\newblock {Stability of Einstein metrics under Ricci flow}
\newblock {In: }\emph{Comm. Anal. Geom.} \textbf{28} (2020) no. 2, 351--394

\bibitem[KV19]{KV18}
\textsc{Kr\"oncke}, Klaus ; \textsc{Vertman}, Boris:
\newblock{Stability of Ricci de Turck flow on Singular Spaces }
\newblock {In: } \emph{Calc. Var. Part. Differ. Equ.}
 \textbf{58} (2019) no. 2, 74 

\bibitem[KV21]{KV19}
\textsc{Kr\"oncke}, Klaus ; \textsc{Vertman}, Boris:
\newblock{Perelman's Entropies for Manifolds with conical Singularities}
\newblock {In: } \emph{Trans. Amer. Math Soc.} \textbf{374} (2021), no. 4, 2873--2908 

\bibitem[KP20]{KP20}
\textsc{Kr\"oncke}, Klaus ; \textsc{Lindblad Petersen}, Oliver:
\newblock{$L^p$-stability and positive scalar curvature rigidity for Ricci-flat ALE manifolds}
\newblock Preprint, arXiv:2009.11854

\bibitem[Lic61]{Lic61}
\textsc{Lichnerowicz}, Andr{\'e}:
\newblock Propagateurs et commutateurs en relativit{\'e} g{\'e}n{\'e}rale.
\newblock {In: }\emph{Publications Math{\'e}matiques de l'IH{\'E}S} \textbf{10} (1961),
no. 1, 5--56


%

%


\bibitem[Oba62]{Ob62}
\textsc{Obata}, Morio:
\newblock {Certain conditions for a Riemannian manifold to be isometric with a
  sphere.}
\newblock {In: }\emph{J. Math. Soc. Japan} \textbf{14} (1962), 333--340

\bibitem[Pac13]{Pac13}
\textsc{Pacini}, Tommaso:
\newblock {Desingularizing isolated conical singularities: uniform estimates via weighted Sobolev spaces.}
\newblock {In: }\emph{Comm. Anal. Geom.} \textbf{21} (2013), no.\ 1 105--170

%

%
%

\bibitem[Ver21]{Ver16}
\textsc{Vertman}, Boris;
\newblock{Ricci de Turck flow on singular manifolds}
\newblock {In: } \emph{J. Geom. Anal.} \textbf{31} (2021), no. 4, 3351-–3404




\end{thebibliography}
\end{document}